\documentclass[twoside,a4paper,11pt]{amsart}

\RequirePackage[normalem]{ulem}
\RequirePackage{color}\definecolor{RED}{rgb}{1,0,0}\definecolor{BLUE}{rgb}{0,0,1}

\usepackage[symbol]{footmisc}
\usepackage{epigraph}
\usepackage{geometry}
\usepackage{amsmath}
\usepackage{amssymb}
\usepackage{amsthm}
\usepackage{mathrsfs}
\usepackage{mathtools}
\usepackage{esint}
\usepackage{graphicx}
\usepackage{color}
\usepackage{tikz}
\usetikzlibrary{patterns}
\usepackage[english]{babel}

\usepackage{hyperref}
\hypersetup{
    backref=true, 
    pagebackref=true,
    plainpages=false,
    bookmarks=true,         
    unicode=false,          
    pdftoolbar=true,        
    pdfmenubar=true,        
    pdffitwindow=true,      
    pdftitle={My title},    
    pdfauthor={Author},     
    pdfsubject={Subject},   
    pdfcreator={Creator},   
    pdfproducer={Producer}, 
    pdfkeywords={keywords}, 
    pdfnewwindow=true,      
    colorlinks=true,       
    linkcolor=blue,          
    citecolor=blue,        
    filecolor=blue,      
    urlcolor=magenta,          
    urlbordercolor=0 1 1
}

\newcommand{\E}{\mathbb{E}}

\newcommand{\N}{\mathbb{N}}

\newcommand{\R}{\mathbb{R}}

\newcommand{\cF}{\mathcal{F}}

\newcommand{\cH}{\mathcal{H}}

\newcommand{\cL}{\mathcal{L}}
\newcommand{\cN}{\mathcal{N}}

\newcommand{\cS}{\mathcal{S}}

\newcommand{\tE}{\tilde{E}}
\newcommand{\tF}{\tilde{F}}

\newcommand{\tY}{\tilde{Y}}

\newcommand{\hE}{\hat{E}}

\def\diam{\operatorname{diam}}

\newcommand{\Lip}{\mbox{\rm Lip}}
\newcommand{\ind}{\mbox{\rm Ind}}

\renewcommand{\div}{\mbox{\rm div}}

\newcommand{\bcup}{\bigcup}
\newcommand{\bcap}{\bigcap}
\newcommand{\epty}{\emptyset}

\newcommand{\sm}{\setminus}
\newcommand{\lgl}{\langle}
\newcommand{\rgl}{\rangle}
\newcommand{\pa}{\partial}
\newcommand{\con}{\subset}

\newcommand{\res}{
	\,\raisebox{-.127ex}{\reflectbox{\rotatebox[origin=br]{-90}{$\lnot$}}}\,
} 

\newcommand{\beq}{\begin{equation}}
\newcommand{\eeq}{\end{equation}}
\newcommand{\beqs}{\begin{equation*}}
\newcommand{\eeqs}{\end{equation*}}

\newcommand{\ep}{\varepsilon} 
\newcommand{\ga}{\gamma}

\newcommand{\al}{\alpha}
\newcommand{\de}{\delta}

\newcommand{\ro}{\rho}
\newcommand{\si}{\sigma}
\newcommand{\te}{\theta}
\newcommand{\De}{\Delta}
\newcommand{\Ga}{\Gamma}

\theoremstyle{plain}
\newtheorem{thm}{Theorem}[section] 

\theoremstyle{plain}

\theoremstyle{plain}
\newtheorem{prop}[thm]{Proposition}

\theoremstyle{plain}
\newtheorem{lemma}[thm]{Lemma}

\theoremstyle{plain}

\theoremstyle{definition}
\newtheorem{defn}[thm]{Definition} 

\theoremstyle{definition}
\newtheorem{remark}[thm]{Remark}

\theoremstyle{definition}

\theoremstyle{definition}
\newtheorem{thmintro}[thm]{Theorem}

\makeatletter
\newtheorem*{rep@theorem}{\rep@title}
\newcommand{\newreptheorem}[2]{%
	\newenvironment{rep#1}[1]{%
		\def\rep@title{#2 \ref{##1}}%
		\begin{rep@theorem}}%
		{\end{rep@theorem}}}
\makeatother
\newreptheorem{theorem}{Theorem}

\title[Connected perimeter of planar sets]{Connected perimeter of planar sets}

\author{Fran\c{c}ois Dayrens}
\address{Univ Lyon, Universit\'e Claude Bernard Lyon 1, CNRS UMR 5208, Institut Camille Jordan, 43 Bd du 11 novembre 1918, 69622 Villeurbanne Cedex, France}
\email{dayrens@math.univ-lyon1.fr}

\author{Simon Masnou}
\address{Univ Lyon, Universit\'e Claude Bernard Lyon 1, CNRS UMR 5208, Institut Camille Jordan, 43 Bd du 11 novembre 1918, 69622 Villeurbanne Cedex, France}
\email{masnou@math.univ-lyon1.fr}

\author{Matteo Novaga}
\address{Dipartimento di Matematica, Universit\`a di Pisa, Largo Bruno Pontecorvo 5, 56127 Pisa, Italy.}
\email{matteo.novaga@unipi.it}

\author{Marco Pozzetta}
\address{Dipartimento di Matematica, Universit\`{a} di Pisa, Largo Bruno Pontecorvo 5, 56127 Pisa, Italy.}
\email{pozzetta@mail.dm.unipi.it}

\begin{document}

\begin{abstract}
We introduce a notion of {\it connected} perimeter for planar sets defined as the lower semicontinuous envelope of perimeters of approximating sets which are measure-theoretically connected. A companion notion of {\it simply connected} perimeter is also studied. 
We prove a representation formula which links the connected perimeter, the classical perimeter, and the length of suitable Steiner trees. We also discuss the application of this notion to the existence of solutions to a nonlocal minimization problem with connectedness constraint.
\end{abstract}

\maketitle

\vspace{-10mm}
\tableofcontents	
\vspace{-0.5cm}

\vspace{0.5cm}


\section{Introduction}

Various problems in biology, physics, engineering, image processing, or computer graphics can be modeled as shape optimization problems whose solutions are connected sets which minimize a specific geometric energy. Typical examples are three-dimensional red blood cells whose boundaries minimize the second-order Helfrich energy~\cite{Seifert}, two-dimensional soap films which are connected solutions to the Plateau problem, conducting liquid drops which minimize a non-local perimeter~\cite{GoNoRu15}, or one-dimensional compact connected sets which have minimal length and contain a given compact set, i.e., solutions to the so-called Steiner problem~\cite{GiPo68,PaSt13}. 

This paper is devoted to the case where the sets are planar and the geometric energy is a suitable relaxation of the perimeter of a set. A convenient notion of perimeter in a variational context is the well-known Caccioppoli's perimeter (see for instance~\cite{AmFuPa00}), which can be defined for sets whose characteristic function is only locally integrable,
and it is finite on the so-called finite perimeter sets.
The classical topological notion of connectedness is not appropriate in this generality because adding or removing Lebesgue-negligible sets may change the connectedness of a set without changing its perimeter. To circumvent this problem, a notion of measure-theoretic connectedness (and simple connectedness) for sets of finite perimeter has been introduced in~\cite{AmCaMaMo01}. The purpose of this paper is to study a {\it $L^1$-relaxed connected perimeter}, i.e., a suitable notion of perimeter for planar sets which are $L^1$-limits of measure-theoretically connected sets.  As will be clear later, there is a strong connection between this notion of connected perimeter and the Steiner problem.

To the best of our knowledge, this is the first contribution proposing a theoretical characterization of {connected perimeter}. However,
motivated by the numerical applications, there have been several contributions on the approximation of such perimeter, or on the approximation of other (sometimes higher-order) related energies, see for instance \cite{DoLeWo17,DoMuRo,DoNoWiWo18,BoLeSa15,BoLeMi18}.

\textcolor{white}{text}

We will constantly use in this work the notion of set of finite perimeter and its main properties, for which we refer to \cite{AmFuPa00}. In Subsection \ref{sub1} we recall the definitions and the results we will need about the concepts of indecomposable and simple set; roughly speaking these are the analogues in the context of sets of finite perimeter of the notions of connected and simply connected set.
Once these definitions are stated, we can introduce the following notion of perimeter. 
If $E\con\R^2$ is measurable, we set
\beq
	P_C(E) = \begin{cases}
					P(E) & \text{ if $E$ is indecomposable,}\\
					+\infty & \text{ otherwise,}
	\end{cases}	
\eeq
and
\beq
P_S(E) = \begin{cases}
	P(E) & \text{ if $E$ is simple,}\\
	+\infty & \text{ otherwise.}
\end{cases}	
\eeq
We deduce by relaxation the {\it connected perimeter} of a set $E$:
\beq 
	\overline{P_C}(E) = \inf\left\{ \liminf_{n \to +\infty} P_C(E_n) \ : \ E_n \to E \text{ in } L^1  \right\},
\eeq
and its {\it simply connected perimeter}:
\beq 
	\overline{P_S}(E) = \inf\left\{ \liminf_{n \to +\infty} P_S(E_n) \ : \ E_n \to E \text{ in } L^1  \right\}.
\eeq
where $E_n \to E \text{ in } L^1$ means the convergence in $L^1$ of the associated characteristic functions.

By the lower semi-continuity of Caccioppoli's perimeter, we obviously have that $\overline{P_C}(E)=P(E)$ if $E$ is indecomposable, and $\overline{P_S}(E)=P_S(E)$ if $E$ is simple.

\medskip
The analog of $P_C$ and $P_E$ for smooth sets in the classical framework of connectedness are defined as follows: if $E\con\R^2$ is measurable, we set
\beq
P^r_C(E) = \begin{cases}
	P(E) & \text{ if $E$ is smooth and connected,}\\
	+\infty & \text{ otherwise,}
\end{cases}	
\eeq
and
\beq
P^r_S(E) = \begin{cases}
	P(E) & \text{ if $E$ is smooth and simply connected,}\\
	+\infty & \text{ otherwise.}
\end{cases}	
\eeq
The associated $L^1$-relaxed functionals are denoted as $\overline{P_C^r}$ and $\overline{P_S^r}$, respectively. The first result we will prove is the following identification theorem:

\begin{thmintro} \label{thmA}
	Let $E\con\R^2$ be an essentially bounded set with finite perimeter, i.e., there exists some $\cL^2$-negligible set $A$ such that $E\setminus A$ is bounded. Then\\
	i) if $E$ is simple, there exists a sequence $E_n$ of smooth simply connected sets such that $E_n\to E$ and $P(E_n)\to P(E)$,\\
	ii) if $E$ is indecomposable, there exists a sequence $E_n$ of smooth connected sets such that $E_n\to E$ and $P(E_n)\to P(E)$.\\
	In particular it holds that
	\beq
	\overline{P_C^r}(E) = \overline{P_C}(E), \qquad  \qquad \overline{P_S^r}(E) = \overline{P_S}(E) .
	\eeq
\end{thmintro}

\textcolor{white}{text}

\noindent Our main result concerns a characterization of the connected and simply connected perimeters $\overline{P_C}$,  $\overline{P_S}$ for any set $E\con\R^2$ such that $\cH^1(\pa E\De \pa^* E)=0$, where $\pa^*E$ is the reduced boundary of $E$~\cite{AmFuPa00}. For such a set $E$, $St(E)$ is defined as the Steiner length of $\overline{E^1}$, i.e., the length of a minimal $1$-set connecting all parts of $\overline{E^1}$ (the closure of the set of points with unit $\cL^2$-density with respect to $E$). Similarly $St_c(E)$ denotes the Steiner length of $\overline{E^0}$. Our main result is the following:

\begin{reptheorem}{thmmain}
	Let $E\con\R^2$ be an essentially bounded set with finite perimeter such that $\pa E = \pa^* E \,\cup\, X$ with $\cH^1(X)=0$. We have
	\beqs 
		\overline{P_C}(E) = P(E) + 2St(E),
	\eeqs
	\beqs 
		\overline{P_S}(E) = P(E) + 2St(E) + 2St_c(E). 
	\eeqs
\end{reptheorem}

\textcolor{white}{text}

\noindent We leave for future work an extension of this result to higher dimension (which would require replacing simply connected sets by contractible sets).\\

\noindent The organization of the paper is the following: in Section 2 we recall the basic notions and results about indecomposable and simple sets; we also prove some technical lemmas that we will use in the sequel. In Section 3 we prove Theorem \ref{thmA}. Section 4 is devoted to the proof of Theorem \ref{thmmain}. Finally in Section 5 we discuss an application of the functionals $\overline{P_C}, \overline{P_S}$ to existence issues for a nonlocal minimization problem.

\textcolor{white}{text}


\section{Notation and preliminary results}

\subsection{Notation}

Let $E,F$ be Borel sets of $\R^2$, we introduce the following notations:
\begin{itemize}
	\item $\mathcal{N}_\delta(E) = \{ x \in \R^2 \ | \ d(x,E)< \delta \}$ for any $\de>0$.
	\item $|E|$ is the Lebesgue measure of $E$.
	\item $\cH^k(E)$ is the $k$-dimensional Hausdorff measure of $E$.
	\item $\cL^k$ is the $k$-dimensional Lebesgue measure.
	\item $d_\cH$ is the Hausdorff distance.
	\item $E = F  \mbox{ mod } \nu$ if $\nu$ is a positive measure and $\nu(E\Delta F)=0$, where $E\Delta F$ 
	is the symmetric difference between $E$ and $F$, that is, $E\Delta F = (E\sm F) \cup (F \sm E)$.
	\item $E^t$ is the set of points of $E$ with a density equal to $t$, i.e.,
	\[ E^t = \left\{ x \in \R^2 \ | \ \lim_{r \to 0} \frac{|E\cap B_r(x)|}{|B_r(x)|} = t \right\} \]
	where $B_r(x)$ is the open ball with center $x$ and radius $r$.
	\item $(\ga)$ is the image of a curve $\ga:[a,b]\to\R^2$.
	\item $\pa E$, $\mathring{E}$ and $\overline{E}$ are the classical topological boundary, interior and closure of $E$, respectively.
	\item $\pa^*E:=\R^2\sm (E^0\cup E^1)$ is the essential boundary of $E$.
	\item $|\mu|$ is the total variation measure of a Radon measure $\mu$.
	\item $D\chi_E$ is the gradient measure of a characteristic function $\chi_E\in BV$.
	\item $\cF E$ is the reduced boundary of a set of finite perimeter $E$, i.e.,
		\[\cF E=\bigg\{x\in\R^2\,\,|\,\,\exists\nu_E(x):=\lim_{r\searrow0}\frac{D\chi_E(B_r(x))}{|D\chi_E|(B_r(x))}\in S^1\bigg\}. \]
	\item $\sharp A$ is the cardinality of a set $A$.
	\item $\ind_\ga(x)$ for $\ga:[a,b]\to\R^2$ a closed curve and $x\not\in(\ga)$ is the index of $x$ with respect to $\ga$.
	\item $A\simeq B$ means that $A$ and $B$ are homeomorphic.
\end{itemize}

\subsection{Connectedness for sets of finite perimeter}\label{sub1}

A theory of measure-theoretic connectedness for sets of finite perimeter was developed thoroughly in~\cite{AmCaMaMo01}. We recall some useful facts for the particular case of planar sets.

\begin{defn}
	Let $E\con\R^2$ be a set with finite perimeter. We say $E$ is \emph{decomposable} if there exist two measurable non negligible sets $A$ and $B$ such that
	\[ E = A \cup B \qquad \text{and} \qquad P(E) = P(A) + P(B) .\]
	We say that a set is \emph{indecomposable} if it is not decomposable.
\end{defn}

\begin{remark}
	An open connected set $E$ with $\cH^1(\pa E)<+\infty$ is indecomposable.
\end{remark}

\noindent The following decomposition result holds:
\begin{thm}[Decomposition Theorem~\cite{AmCaMaMo01}] \label{thm1}
	Let $E\con\R^2$ be a set of finite perimeter. There exists a unique family of sets $(E_i)_{i \in I}$ with $I$ at most countable such that\\
		i) $|E_i|>0$,\\
		ii) $P(E)=  \sum_{i\in I} P(E_i)$,\\
		iii) $\cH^{1} \left( E^1 \sm \bigcup_{i\in I} E_i^1 \right) = 0$,\\
		iv) $E_i$ is indecomposable and maximal, i.e., for all indecomposable set $F\subset E$, there exists $i\in I$ such that $F \subset E_i$. 
\end{thm}

\noindent The sets $E_i$ in Theorem \ref{thm1} are called the \emph{M-connected components} of $E$.

\begin{defn}
	Let $E\con\R^2$ be a set with finite perimeter.\\
		i) If $E$ is indecomposable then a \emph{hole} of $E$ is a $M$-connected component of $\R^2 \sm E$ with finite measure.\\
		ii) If $E$ is indecomposable then the \emph{saturation} of $E$, denoted $sat(E)$, is the union of $E$ and its holes.\\
		iii) If $E$ is decomposable then its saturation $sat(E)$ is given by the union of the saturation of its $M$-connected components $E_i$, i.e.,
		\[ sat(E) = \bigcup_{i\in I} sat(E_i) .\]
		iv) $E$ is called \emph{saturated} if $E=sat(E)$.\\
		v) $E$ is called \emph{simple} if it is saturated and indecomposable.\\
		vi) if $|E|<+\infty$, the unique M-connected component of $\R^2\sm E$ with infinite measure is the \emph{exterior} $ext(E)$ of $E$.\\
\end{defn}

\begin{defn}
	A subset $J$ of $\R^2$ is a \emph{Jordan boundary} if there exists a simple set $E$ such that $J = \pa^* E  \mbox{ mod } \cH^{1}$. Such a set $E$ is necessarily unique $\mbox{mod }\cL^2$, $E$ is called the \emph{interior} of $J$, and it is denoted by $int(J)$.
\end{defn}

\noindent The following result describes the decomposition of the reduced boundary of a set of finite perimeter in terms of a collection of nested external Jordan boundaries $J_i^+$ and internal Jordan boundaries $J_k^-$ (see~Figure \ref{figSetsHoles}). In order to simplify the statement, the class of Jordan boundaries is enlarged by introducing a \emph{formal} Jordan boundary $J_\infty$ whose interior is $\R^2$ and another \emph{formal} Jordan boundary $J_0$ whose interior is empty. We also set $\cH^1(J_\infty)=\cH^1(J_0)=0$. We will denote by $\cS$ this extended class of Jordan boundaries. This allows to consider sets with finite and infinite measure and we can always assume that the list of components (or holes of the components) given by the following theorem is infinite, possibly adding to it infinitely many $int(J_0)$.\\

\noindent With such definitions, Theorem \ref{thm1} can be refined in the following way.

\begin{thm}[Boundary decomposition~\cite{AmCaMaMo01}]\label{thm2}
	Let $E\con\R^2$ be a set of finite perimeter. Then there exists a unique decomposition $\mbox{mod }\cH^1$ of $\pa^* E$ into Jordan boundaries
	\beq \label{eq1}
	 \{ J^+_i, J^-_k \ | \ i,k \in \N \}\con\cS
	\eeq
	with possibly $\cH^{1}(J^{\pm}_j)=0$, i.e., $int(J^{\pm}_j)=\emptyset$ or $\R^2 \mbox{ mod } \cH^{1}$ such that:\\
		i) $int(J^+_i)$ and $int(J^+_j)$ are either disjoint or one subset of the other.\\
		ii) $int(J^-_k)$ and $int(J^-_j)$ are either disjoint or one subset of the other.\\
		iii) For all $k$ there exists $i$ such that $int(J^-_k)\subset int(J^+_i)$.\\
		iv) If $int(J^+_j)\subset int(J^+_i)$ for some $i \neq j$ then there exists $k$ such that
		$int(J^+_j)\subset int(J^-_k) \subset int(J^+_i)$.\\
		v) If $int(J^-_j)\subset int(J^-_k)$ for some $j \neq k$ then there exists $i$ such that
		$int(J^-_j)\subset int(J^+_i) \subset int(J^-_k)$.\\
		vi) $P(E)=  \sum_i \cH^{1}(J^+_i) + \sum_k \cH^{1}(J^-_k)$.\\
		vii) for all $i$, we denote $L_i = \{ k \ | \ int(J^-_k) \subset int(J^+_i) \}$ and
		\[ Y_i = int(J^+_i) \sm \bigcup_{k \in L_i} int(J^-_k) . \]
		The sets $(Y_i)_i$ are pairwise disjoint and indecomposable and
		$E =  \bigcup_i Y_i \mbox{ mod }\cL^2$.
\end{thm}

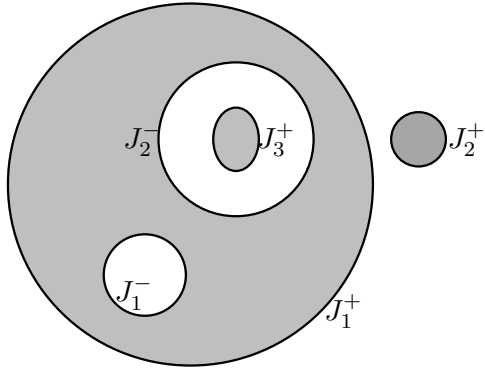
\begin{figure}[h]
	\begin{center}
		\begin{tikzpicture}[scale=0.6]
		\filldraw[fill=gray!50, line width=0.3mm, draw=black, even odd rule]
		(0,0) circle (4cm)  (1,1) circle (1.7cm)  (-1,-2) circle (0.9cm)  (1,1)ellipse (0.5cm and 0.7cm);
		\filldraw[fill=gray!70, line width=0.3mm, draw=black, even odd rule]
		(5,1) circle (0.6cm);
		\path[font=\normalsize]
		(4,-2.8)node[left]{$J^+_1$};
		\path[font=\normalsize]
		(6.7,1)node[left]{$J^+_2$};
		\path[font=\normalsize]
		(-0.6,-2.4)node[left]{$J^-_1$};
		\path[font=\normalsize]
		(-0.41,1)node[left]{$J^-_2$};
		\path[font=\normalsize]
		(2.52,1)node[left]{$J^+_3$};
		\end{tikzpicture}
	\end{center}
	\caption{Decomposition of the boundary of a finite perimeter set in $\R^2$ using Jordan boundaries (i.e., boundaries of simple sets) denoted as in Theorem \ref{thm2}.}\label{figSetsHoles}
\end{figure}

\begin{prop}[Boundary of a simple planar set, \cite{AmCaMaMo01}]\label{prop1}
	Let $E\con\R^2$ be a simple set with $0<|E|<+\infty$. Then there exists a Jordan curve $\Gamma$ such that $\pa^* E = \Gamma  \mbox{ mod } \cH^1$. Moreover, $\Gamma$ admits a Lipschitz parametrization and $P(E) = \cH^1 (\Gamma)$.
\end{prop}

\begin{remark}
	From Proposition \ref{prop1} we see that the Jordan boundaries given in Theorem \ref{thm2} can be parametrized by Lipschitz Jordan curves (i.e., Jordan curves which admit Lipschitz parameterizations). Moreover, if $E$ is simple then the family in \eqref{eq1} is reduced to only one nontrivial curve $J^+_0$.
\end{remark}

\begin{remark}\label{rem-essbound}
	If $E$ is a set of finite perimeter, it holds that
	\beq
		E\mbox{ indecomposable,}\quad |E|<+\infty \qquad \Rightarrow \qquad E \mbox{ essentially bounded.}
	\eeq
	In fact, letting $F=sat(E)$, we have from \cite{AmCaMaMo01} that $F$ is simple. Since $|E|<+\infty$, by definition of saturation we have that the exterior $ext(E)$ is disjoint $\mbox{mod }\cL^2$ from $F$. Hence $|F|<+\infty$, thus $F$ is equivalent to $int(\Ga)$ for a Lipschitz Jordan curve $\Ga$, then $F$ is essentially bounded. Since $E\con F$, the set $E$ is essentially bounded as well. 
\end{remark}

\begin{remark} \label{rem1}
	We recall that if $E\con\R^2$ is a set of finite perimeter such that $\pa E = \cF E \mbox{ mod }\cH^1$, then
	\beq
	E^1=\mathring{E} \quad\mbox{ mod }\cH^1, \qquad\qquad E^0=\R^2\sm \overline{E} \quad\mbox{ mod }\cH^1.
	\eeq
\end{remark}

\noindent We finish this part with some consequences we will need in the sequel.

\begin{lemma} \label{lem1}
	Let $E\con\R^2$ be a set of finite perimeter with $0<|E|<+\infty$. Suppose that $E$ is indecomposable, then $E$ is essentially bounded and in the notation of Theorem \ref{thm2} it holds that
	\beq
		\{J^+_i\}_{i\in\N}=\{J^+_0\}, \qquad \forall i:\quad int(J^-_i)\con int(J^+_0) , \qquad \forall i\neq k:\quad |int(J^-_i) \cap int(J^-_k)|=0 
	\eeq
	up to relabeling and dropping curves $J^\pm_j$ with $|int(J^\pm_j)|=0$. In particular $E=Y_0:=int(J^+_0)\sm \cup_i int(J^-_i) \mbox{ mod }\cL^2$.
\end{lemma}

\begin{proof}
	Let $\{J^\pm_i:i\in\N\}$ be the family of Jordan boundaries given by Theorem \ref{thm2}. Up to dropping a subfamily of such curves, we can assume $|int(J^\pm_i)|>0$ that for any $i$. Then for any $J^\pm_i$ there exist finitely many indexes $j$ such that $int(J^\pm_i)\con int(J^\pm_j)$, in fact by isoperimetric inequality for any such $j$ we have $\cH^1(J^\pm_j)\ge C(i)>0$ and $E$ has finite perimeter. Therefore, using also property $iii)$ of Theorem \ref{thm2}, there exists at least a curve $J^+_{i_0}$ such that $int(J^+_{i_0})$ is maximal with respect to inclusion. For any $k\neq i_0$ the sets $int(J^+_{i_0})$ and $int(J^+_k)$ are either disjoint or one subset of the other $\mbox{mod }\cL^2$. Being $E$ indecomposable and since $int(J^+_{i_0})$ is maximal with respect to inclusion, we conclude that any set $int(J^+_k)$ is contained in $int(J^+_{i_0})$. From now on we relabel $J^+_{i_0}$ into $J^+_0$.\\
	Now if there exists a curve $J^+_k\neq J^+_0$, by property $iv)$ of Theorem \ref{thm2} we would get some $J^-_j$ such that the set $(int(J^+_0)\sm int(J^-_j))\cup int(J^+_k)$ is decomposable, which contradicts the fact that $E$ is indecomposable. Therefore $\{J^+_i\}_{i\in\N}$ is the singleton $\{J^+_0\}$.\\
	Finally by property $v)$ of Theorem \ref{thm2} we get that $|int(J^-_i) \cap int(J^-_k)|=0$ for any $i\neq k$.
\end{proof}

\begin{lemma} \label{lem2}
	Suppose $E\con\R^2$ is indecomposable with $0<|E|<+\infty$. Suppose also that the family of Jordan curves $\{J^+_0,J^-_i\}$ with non-trivial interior given by Lemma \ref{lem1} is finite. Then $E$ is essentially bounded and
	\beq \label{eq2}
		 E=\bigg\{ p\in\R^2\sm((J^+_0)\cup_i (J^-_i))\,\,|\,\, Ind_{J^+_0}p+\sum_i Ind_{J^-_i}p \equiv 1 \mbox{ mod }2  \bigg\} \quad\mbox{mod }\cH^1.
	\eeq
	In particular $E$ is equivalent to an open set. Moreover, using the representative of $E$ in \eqref{eq2}, we have that
	\beq
		P(E)=\cH^1(\pa E).
	\eeq
\end{lemma}

\begin{proof}
	Denoting by $J^\pm_i$ also a constant velocity Lipschitz parametrization of $(J^\pm_i)$ for any $i\ge0$, one has that the set $\{p\in\R^2\sm(J^\pm_i)\,|\,\ind_{J^\pm_i}(p)\equiv1 \mbox{ mod }2\}$ is a representative for $int(J^\pm_i)$. Hence writing $E=int(J^+_0)\sm\cup_{i\ge1} int(J^-_i)$ by Lemma \ref{lem1}, using also Remark \ref{rem1}, \eqref{eq2} immediately follows.\\
	Now we observe that, using the notation of Lemma \ref{lem1}, we have that $\cH^1((J^-_i)\cap (J^-_k))=0$ for any $i\neq k$. In fact if by contradiction we assume that $\cH^1((J^-_i)\cap (J^-_k))>0$ for some $i\neq k$, since under our hypotheses the holes are simple sets, we would have from \cite{AmCaMaMo01} that $U=int(J^-_i)\cup int(J^-_k)$ is indecomposable (hence M-connected). Thus $U$ would be a hole of $E$, but this contradicts the uniqueness of the decomposition of Theorem \ref{thm2}. Similarly we conclude that $\cH^1((J^-_i)\cap (J^+_0))=0$. Then we can use Lemma 2.8 in \cite{Po19} to get that
	\[ 
		\cH^1\res \cF E = \cH^1 \res J^+_0 + \sum_i \cH^1\res J^-_i = \cH^1\res \bigg( J^+_0\cup\bcup_i J^-_i \bigg).
	\]
	Since $J^+_0\cup\bcup_i J^-_i$ is closed, it coincides with $\pa E$, and hence we have $\cH^1(\cF E)=P(E)=\cH^1(\pa E)$.
\end{proof}

\begin{lemma}\label{lem3}
	Suppose $E\con\R^2$ is indecomposable with $0<|E|<+\infty$. Then
	\beq
		2\diam E^1 \le P(E).
	\eeq 
\end{lemma}

\begin{proof}
	By Lemma \ref{lem1} we can write $E=int(J^+_0)\sm \cup_i int(J^-_i)$. We know from \cite{AmCaMaMo01} that the set $F=sat(E)$ is simple, so we can identify it with $int(\Ga)$ for a Jordan Lipschitz curve $\Ga$ with $P(F)=\cH^1(\Ga)$. Also $int(\Ga)= \{p\in\R^2\sm\Ga\,\,|\,\, Ind_\Ga p=1 \}$. By construction $\Ga=J^+_0$, thus we have that $\diam E^1=\diam F^1=\diam F$ and for any $x\neq y$ with $x,y\in F$ it holds that
	\[ P(E)\ge P(F)=\cH^1(\Ga)\ge 2|x-y|. \]
	Passing to the supremum on $x\neq y$ with $x,y \in F$ we get the estimate.
\end{proof}

\noindent For the convenience of the reader, we finally recall here a useful result.

\begin{thm}[\cite{Sc15}] \label{thmSchmidt}
	Let $Y$ be an open bounded set in $\R^d$ such that $P(Y) = \cH^{d-1}(\pa Y)$. Then for every $\delta>0$ there exists a smooth set $Y_\delta$ satisfying:\\
		i) $Y_\delta \subset Y$,\\
		ii) $Y \sm Y_\delta \subset \mathcal{N}_\delta(\pa Y) \cap \mathcal{N}_\delta(\pa Y_\delta)$,\\
		iii) $P(Y_\delta) \le P(Y) + \delta$,\\
	with $\mathcal{N}_\delta(A) = \{ x \in \R^d \ | \ d(x,A)< \delta \}$ for any set $A$.
\end{thm}

\begin{remark}\label{rem2}
	As long as a finite perimeter set $E$ is equivalent to an open set $Y$ satisfying $P(Y) = \cH^{d-1}(\pa Y)$, then Theorem \ref{thmSchmidt} is applicable. In particular one can apply Theorem \ref{thmSchmidt} in any of the following cases:\\
	i) $E$ simple with $|E|<+\infty$ (by Proposition \ref{prop1}),\\
	ii) $E$ indecomposable with $|E|<+\infty$ with a finite number of holes (by Lemma \ref{lem2}).
\end{remark}

\textcolor{white}{.}

\subsection{The Steiner problem} We provide some basic definitions and results that we will use in the sequel.

\begin{defn}
	Let $K$ be a compact subset of $\R^2$. The Steiner problem associated with $K$ is the optimization problem
	\begin{equation} \label{steinerinf}
	\sigma(K) = \min \{ \cH^1(S),\; \ K\cup S \text{ is connected} \}.
	\end{equation}
	$\sigma(K)$ is called the Steiner length of $K$.
\end{defn}

\noindent It must be emphasized that the infimum in \eqref{steinerinf} is a minimum. We collect below some definitions and qualitative properties of the solutions to the Steiner problem, see~\cite{PaSt13} .

\begin{defn}
		We say that $S\con\R^2$ is a \emph{tree} if $S$ is an unoriented planar graph without loops composed of a set $V$ of vertices and a set $A$ of disjoint segments with endpoints in $V$. The \emph{degree} of a vertex $v\in V$ is the number of edges incident to $v$ (possibly equal to $+\infty$).\\
		A vertex with degree $1$ is called \emph{endpoint}. A vertex with degree $>1$ is called \emph{branching point}. A vertex with degree $3$ is called \emph{triple joint}.\\
		The set $S$ is a \emph{finite tree} if $V$ is finite (i.e., $S$ has a finite number of connected components and branching points).
\end{defn}

\noindent We recall the following results which have been proved in \cite[Theorems 5.1, 7.6, 7.4, 7.3]{PaSt13}.

\begin{thm} 
	Let $K\con\R^2$ be a compact set and let $S$ be a minimizer of \eqref{steinerinf} such that $\cH^1(S) < +\infty$. Then\\
		i) $K\cup S$ is compact,\\
		ii) $S\sm K $ has at most countably many connected components and each of them has positive $\cH^1$ measure,\\
		iii) $S$ contains no loops,\\
		iv) the (topological) closure of every connected component of $S$ is a tree with endpoints on $K$, with at most one endpoint on each connected component of $K$,\\
		v) $S\sm \cN_\ep(K)$ is a finite tree for almost every $\ep>0$,\\
		vi) if $K$ is finite, then $S$ is a finite tree and every vertex is either a point of $K$ or a triple joint.
\end{thm}

\begin{defn} \label{def1}
	Let $E$ be an essentially bounded set of finite perimeter in $\R^2$ such that $\pa E = \pa^* E  \mbox{ mod } \cH^1$. Let $S$ be a Steiner tree for $\overline{E^1}$ and $S^c$ a Steiner tree for $\overline{E^0}$. We denote $St(E)=\cH^1(S)$ and $St_c(E) = \cH^1(S^c)$.
	\end{defn}

\begin{remark}
	Since in the above definition the set $\overline{E^0}$ is not compact, the Steiner problem on $\overline{E^0}$ is defined on the compact set $\overline{B_R(0)}\cap \overline{E^0}$ for $R$ sufficiently large so that $|E\sm B_{\frac{R}{2}}(0)|=0$. The quantity $St_c(E)$ is clearly independent of the choice of any such $R$.
\end{remark}
\textcolor{white}{.}


\section{Equivalence of the relaxations}

\noindent Recalling the definitions seen in the introduction of $P_C$, $P_C^r$, $P_S$, $P_S^r$, and their associated  $L^1$-relaxations, we now prove the following result:

\begin{thm} \label{thmsmoothsets}
	Let $E\con\R^2$ be an essentially bounded set with finite perimeter. It holds that
	\beq
	 \overline{P_C^r}(E) = \overline{P_C}(E), \qquad  \qquad \overline{P_S^r}(E) = \overline{P_S}(E) .
	\eeq
\end{thm}

\begin{proof}
	Let us start by proving that $\overline{P_S^r}(E) = \overline{P_S}(E)$. By a diagonal argument it is enough to prove that given a simple set $E$, we can approximate $E$ in the $L^1$ sense with a sequence of simply connected smooth sets with perimeter converging to $P(E)$.\\
	So let $E$ be a simple set. We can identify $E$ with the open set $int(J^+)$ where $J^+$ denotes the Jordan boundary of $E$, which is a Lipschitz curve with $\cH^1(J^+)=P(E)$. By Theorem \ref{thmSchmidt} and Remark \ref{rem2} there exists a sequence $E_\ep$ of smooth set such that $P(E_\ep) \le  P(E) + \ep$. Also $E\sm E_\ep \con \cN_\ep(\pa E)\cap \cN(\pa E_\ep)$ and $E$ is simple, then the boundary of any connected component of $E_\ep$ is contained in $\cN_\ep(\pa E)$. Then there exists a connected component $\tilde{E}_\ep$ of $E_\ep$ such that
	\[ E \sm  \mathcal{N}_\ep(\pa E) \subset \tilde{E}_\ep.\]
	The set $F_\ep=sat(\tilde{E}_\ep)$ is a smooth and simply connected set contained in $E$ with $\pa F_\ep\con \cN_\ep(\pa E)$. We have
	\[ P(F_\ep) \le  P(E_\ep) \le P(E) + \ep\]
	and
	\[ E \Delta F_\ep\subset \mathcal{N}_\ep(\pa E) . \]
	Since by Lemma \ref{lem2} we have that $\pa E = \pa^* E  \mbox{ mod } \cH^1$ is rectifiable, using Theorem 3.2.39 in \cite{Fe69} we get
	\[ \lim_{\ep\to 0} \frac{| \mathcal{N}_\ep(\pa E) |}{ 2\ep} =  P(E), \]
	and then $F_\ep \to E$ in $L^1$. By the lower semicontinuity of the perimeter we obtain $P(F_\ep) \to P(E)$.\\

	\noindent We follow a similar strategy in the case of $P_C$. The goal is still to approximate an indecomposable set $E$ with smooth connected sets having perimeter converging to $P(E)$.
	In the notation of Lemma \ref{lem1} we can identify $E$ with
	\[ Y_0 \sm \bigcup_{j \in J} T_j. \]
	The sets $Y_0,T_j$ are simple, bounded, and open for any $j$. Let $\ep>0$, we define:\\
	i) $J_\ep= \{ j \in J \ | \ |T_j|>\ep\} $.\\
	ii) $Y_{0,\ep}$ is an approximation from outside of $Y_0$ constructed as follows. As $Y_0$ is bounded and simple, we can approximate its complement set in some large ball and then perform the approximation from within of such complement as given by Theorem \ref{thmSchmidt} with $\de=\ep$ (see also Remark \ref{rem2}). Adding the complement of the ball, we obtain a smooth set $\tilde{Y}_{0,\ep}$. Taking $Y_{0,\ep} = \R^{2} \sm \tilde{Y}_{0,\ep}$, it holds that\\
		\indent	a) $Y_0 \subset Y_{0,\ep}$,\\
		\indent	b) $Y_{0,\ep	} \sm Y_0 \subset \mathcal{N}_\delta(\pa Y_0) \cap \mathcal{N}_\delta(\pa Y_{0,\ep})$,\\
		\indent	c) $P(Y_{0,\ep}) \le P(Y_0) + \ep$.\\
	iii) $T_{j,\ep}$ the approximation from within given by Theorem \ref{thmSchmidt} together with Remark \ref{rem2} of $T_j$ with $\delta = \ep^2$ for $j \in J_\ep$.\\
	iv) $E_\ep= Y_{0,\ep} \sm   \bigcup_{j \in J_\ep} T_{j,\ep}$.\\
	Since $T_j \subset Y_0$ for any $j$ and the $T_j$'s are essentially disjoint, we have that
	\[ \ep \sharp(J_\ep) \le \sum_{j \in J_\ep} |Y_j| \le \sum_{j \in J} |Y_j| \le |Y_0|. \]
	Since $|Y_0| = |sat(E)|<+\infty$, we get that $ \ep \sharp(J_\ep) \le |sat(E)| < +\infty $.\\
	By the same argument used for $P_S$, we may assume that $Y_{0,\ep}$, $T_{j,\ep}$ are smooth simple sets. Hence $E_\ep$ is smooth and connected. We have that
	\beq \label{eq4} \begin{array}{rcl}
	P(E_\ep) & \le & P(Y_{0,\ep}) +   \sum_{j \in J_\ep} P(T_{j,\ep}) \\
	& \le & P(Y_0) + \ep+   \sum_{j \in J_\ep} \big( P(T_j) + \ep^2 \big) \\
	& \le & P(Y_0) + \ep+   \sum_{i \in J_\ep} P(Y_j) + \ep^2 \sharp(J_\ep) \\
	& \le & P(Y_0) + \ep+   \sum_{i \in J} P(Y_j) + \ep^2 \sharp(J_\ep) \\
	& \le & P(E) + \ep+ \ep^2 \sharp(J_\ep) .
	\end{array} \eeq
	Therefore $ \limsup_{\ep\to 0} P(E_\ep) \le P(E)$. Also
	\beq \label{eq3}
	 |E \Delta E_\ep| = |Y_{0,\ep} \sm Y_0| +
	\sum_{j \in J\sm J_\ep} |T_j| +
	\sum_{j \in J_\ep} |T_j \sm  T_{j,\ep}|,
	\eeq
	where the first term comes from the approximation from outside of $sat(E)=Y_0$, the second from the filled small holes, and the third from the approximation of remaining holes from inside. As $ \sum_{j \in J\sm  J_\ep} |T_j|$ is a rest of the absolutely converging series $ \sum_{j \in J} |T_j| < |Y_0|$, we have that $ \sum_{j \in J\sm  J_\ep} |T_j| \to 0 $ as $\ep\to0$.\\
	Also $ |Y_{0,\ep} \sm  Y_0| \le |\mathcal{N}_\ep(\pa Y_0)| \le 4\ep P(E)$ for $\ep$ small enough by Theorem 3.2.39 in \cite{Fe69}. Then $|Y_{0,\ep} \sm Y_0| \to 0$ when $\ep\to 0$.\\
	Analogously for all $j \in J_\ep$ we have
	\[ 
		| T_j\sm T_{j,\ep} | \le |\mathcal{N}_\ep(\pa T_j)| \le \ep^2 P(E) 
	\]
	for $\ep$ small enough depending on $j$. Moreover $  \sum_{j \in J_\ep} |T_j \sm T_{j,\ep}| \le |sat(E)| < +\infty$ then $  \limsup_{\ep\to 0} \sum_{j \in J_\ep} |T_j \sm T_{j,\ep}| < +\infty$. We denote $\ep_h$ a subsequence such that
	\[ 
		\limsup_{\ep\to 0} \sum_{j \in J_\ep} |T_j \sm T_{j,\ep}| =
		\lim_{h \to +\infty} \sum_{j \in J_{\ep_h}} |T_j \sm T_{j,\ep_h}|.
	 \]
	Since $|T_j\sm T_{j,\ep}|\to0$ for any $j$, we have that for all $\eta>0$ there exists $H>0$ such that for all $h>H$, $  \sum_{j \in J_{\ep_h}\sm J_{\ep_H}} |T_j \sm T_{j,\ep_h}| < \eta/2$. The set $J_{\ep_H}$ is finite and $J_{\ep_H} \subset J_{\ep_h}$, then for $h$ large enough
	\[ 
		\sum_{j \in J_{\ep_H}} |T_j \sm T_{j,\ep_h}| \le4 \ep_h^2 \sharp(J_{\ep_H}) P(E) \le 4\ep_h^2\sharp(J_{\ep_h}) P(E) .
	\]
	Choosing $h$ large enough so that $4\sharp(J_{\ep_h})\ep_h^2 P(E) < \eta/2$ we obtain
	\[ 
		\sum_{j \in J_{\ep_h}} |T_j \sm T_{j,\ep_h}| \le
		\sum_{j \in J_{\ep_H}} |T_j \sm T_{j,\ep_h}|
		+ \sum_{j \in J_{\ep_h} \sm J_{\ep_H}} |T_j \sm T_{j,\ep_h}| < \eta .
	\]
	Then
	\[ 
		\limsup_{\ep\to 0} \sum_{j \in J_\ep} |T_j \sm T_{j,\ep}| =
		\lim_{h \to +\infty} \sum_{j \in J_{\ep_h}} |T_j \sm T_{j,\ep_h}| < \eta .
	\]
	Thus, taking $\eta \to 0$, we have
	\[ 
		\sum_{j \in J_\ep} |T_j \sm T_{j,\ep}| \xrightarrow[\ep\to0]{}  0.
	 \]
	Recalling \eqref{eq3} we conclude that
	\[ 
		|E_\ep\Delta E |  \xrightarrow[\ep\to0]{} 0. 
	\]
	By \eqref{eq4} and by lower semicontinuity of the perimeter we have that
	\[ 
		P(E_\ep)   \xrightarrow[\ep\to0]{} P(E) .
	\]
	By a diagonal argument this completes the proof.
\end{proof}

\noindent From the previous proof we remark that the following approximation results hold.

\begin{prop}\label{prop2}
	Let $E\con\R^2$ be an essentially bounded set with finite perimeter. Then\\
	i) if $E$ is simple, there exists a sequence $E_n$ of smooth simply connected sets such that $E_n\to E$ and $P(E_n)\to P(E)$,\\
	ii) if $E$ is indecomposable, there exists a sequence $E_n$ of smooth connected sets such that $E_n\to E$ and $P(E_n)\to P(E)$.
\end{prop}

\noindent Putting together Proposition \ref{prop2} with Theorem \ref{thmsmoothsets}, the proof of Theorem \ref{thmA} is completed.

\textcolor{white}{.}


\section{Representation formulas}

\noindent Recalling the definitions of $St$ and $St_c$ given in Definition \ref{def1}, we now prove the main result of this paper.

\begin{thm} \label{thmmain}
	Let $E\con\R^2$ be an essentially bounded set with finite perimeter satisfying $\pa E = \pa^* E  \mbox{ mod } \cH^1$. We have
	\beq \overline{P_C}(E) = P(E) + 2St(E),
	\eeq
	\beq \overline{P_S}(E) = P(E) + 2St(E) + 2St_c(E). 
	\eeq
\end{thm}

\begin{proof}
	The proof follows immediately from Propositions \ref{propliminf}, \ref{proplimsup1}, \ref{proplimsup2}, and \ref{propapprox}, which will be proved in the following subsections.
\end{proof}

\textcolor{white}{.}

\subsection{Lim\,inf inequality}

\begin{prop}\label{propliminf}
	Let $E$ be an essentially bounded set of finite perimeter satisfying $\pa E=\pa^* E \mbox{ mod }\cH^1$. Suppose that $E_n$ is a sequence of sets of finite perimeter converging to $E$ in $L^1$. Then
	\beq \label{eq7}
		P(E)+2St(E)\le \liminf_n P_C(E_n),
	\eeq
	\beq \label{eq5}
		P(E)+2St(E)+2St_c(E) \le \liminf_n P_S(E_n).
	\eeq
\end{prop}

\noindent The proof of Proposition \ref{propliminf} contains some technical lemmas which are proved in the sequel.

\begin{proof}
	We start by proving \eqref{eq5}. Without loss of generality assume that $\sup_n P_S(E_n)<+\infty$ and $\liminf_n P_S(E_n)=\lim_n P_S(E_n)<+\infty$. Let $\ga_n:[0,1]\to\R^2$ be Lipschitz Jordan curves such that $\cH^1((\ga_n) \De \pa^* E_n)=0$. Since $P_S(E_n)$ is uniformly bounded, all $E_n$ are simple sets, thus essentially bounded by Remark~\ref{rem-essbound}. Since they converge in $L^1$ to $E$ which is essentially bounded, they are essentially uniformly bounded. Hence the curves $\ga_n$ are uniformly bounded. The uniform bound on $\cH^1(\ga_n)$ implies the equicontinuity of the family of curves. Thus, by Ascoli-Arzel\`{a} Theorem, the sequence $\ga_n$ converges uniformly up to subsequence to some Lipschitz curve $\ga$.\\
	We define the multiplicity function
	\beqs
		\te:\R^2\to \N\cup\{+\infty\} \qquad \te(x)=\sharp(\ga^{-1}(x)).
	\eeqs
	By the area formula it follows that $\te$ is finite $\cH^1$-ae on $\R^2$.
	
	\begin{lemma} \label{lem4}
		Let $E,E_n,\ga,\ga_n$ be as in the proof of \eqref{eq5}. Suppose $\ga(t)=x \in \R^2$, $\gamma$ is differentiable at $t$, and $|B_r(x) \cap E|=0$ for some $r>0$. Then $\theta(x) \ge 2$.
	\end{lemma}
	
	\noindent Lemma \ref{lem4} implies that $\theta(x) \ge 2$ at $\cH^1$-almost every $x \in (\gamma)\sm \pa^* E$.\\
	In fact let $\ga(t)=x\in (\gamma) \sm  \pa^* E$. Up to a $\cH^1$-negligible set, the curve $\ga$ is differentiable at such $t$ and $x\in(\ga)\sm\pa E$. So $x\in \mathring{E}\cup (\R^2\sm\overline{E})$. If $x\in \R^2\sm\overline{E}$ the hypotheses of Lemma \ref{lem4} are satisfies and then $\te(x)\ge 2$. If $x\in \mathring{E}$ one just applies an analogous argument to the set $\R^2\sm E$ in place of $E$.\\
	Also we notice that if $x\in\pa^* E$, then $\te(x)\ge 1$.\\
	In fact we can prove that $\R^2\sm(\ga)\con E^1\cup E^0$. Indeed if $x\in\R^2\sm(\ga)$, by uniform convergence $x\in\R^2\sm(\ga_n)$ for $n$ large, and then there exists $r>0$ such that either
	\beqs
		\frac{|B_r(x)\cap E_n|}{|B_r(x)|}=1  
	\eeqs
	for all large $n$, or
	\beqs
		\frac{|B_r(x)\cap E_n|}{|B_r(x)|}=0
	\eeqs
	for all large $n$. Passing to the limit first in $n$ and then in $r\searrow0$ we see that $x\in E^1\cup E^0$.\\
	
	\noindent By the uniform Lipschitz bound on $\ga_n$, we get that the sequence of derivatives $\ga'_n$ is uniformly bounded in $L^1\cap L^\infty$ and equi-integrable, then, by Dunford-Pettis Theorem, up to subsequence we have that 
	\beq\label{eq6}
	\begin{split}
		\liminf_n P_S(E_n) &= \liminf_n L(\ga_n) \ge L(\ga) = \int_{\R^2} \theta(x)\,d\cH^1(x)=\\
		&= \int_{E^0} \theta(x) \, d  \cH^1(x)
		+  \int_{E^1} \theta(x) \, d  \cH^1(x)
		+ \int_{ \pa^* E} \theta(x)  \,d  \cH^1(x) \ge \\
		&\ge \int_{(\gamma) \cap E^0} 2 \, d  \cH^1(x)
		+  \int_{(\gamma) \cap E^1} 2 \, d  \cH^1(x)
		+ \int_{ \pa^* E} 1  \,d  \cH^1(x)=\\
		&= 2 \cH^1\big( (\gamma)\cap E^0 \big) + 2 \cH^1\big( (\gamma)\cap E^1 \big)  + P(E) =\\
		&= 2 \cH^1\big( (\gamma)\cap ( \R^2\sm \overline{E}) \big)
		+ 2 \cH^1\big( (\gamma)\cap \mathring{E} \big)
		+ P(E),
	\end{split}
	\eeq
	where in the last equality we used Remark \ref{rem1}.
	
	\begin{lemma} \label{lem5}
		Let $E,E_n,\ga,\ga_n$ be as in the proof of \eqref{eq5}. Both the sets $\big( (\gamma)\cap (\R^2 \sm \overline{E}) \big) \cup \overline{E^1}$ and $\big( (\gamma)\cap \mathring{E} \big) \cup \overline{E^0}$ are equivalent $\mbox{mod } \cH^1$ to connected sets.
	\end{lemma}
	
	\noindent By Lemma \ref{lem5}, up to $\cH^1$-negligible sets, the set $\big( (\gamma)\cap (\R^2 \sm \overline{E}) \big) $ is a competitor for the Steiner problem with datum $\overline{E^1}$. Hence $\cH^1\big( (\gamma)\cap ( \R^2\sm \overline{E}) \big)\ge St(E)$. Analogously $\cH^1\big( (\gamma)\cap \mathring{E} \big)\ge St_c(E)$. Hence \eqref{eq6} implies \eqref{eq5}.\\
	
	\noindent Now we prove \eqref{eq7}. Without loss of generality assume that $\liminf_n P_C(E_n)=\lim_n P_C(E_n)<+\infty$. Each $E_n$ is indecomposable and by Theorem \ref{thm2} there exist at most countably many Lipschitz Jordan curves $\ga_{n,i}:[0,1]\to\R^2$ such that $\sum_i L(\ga_{n,i})=P(E_n)$ and $\Lip(\ga_{n,i})=L(\ga_{n,i})\le C$, where $L(\ga_{n,i})$ is the length of $\ga_{n,i}$. By Lemma \ref{lem3} one gets that the sets $E_n$ are uniformly essentially bounded. By Lemma \ref{lem2}, up to relabeling we can assume that $\ga_{n,0}$ is such that $sat(E_n)=int(\ga_{n,0}) \mbox{ mod }\cL^2$, and $L(\ga_{n,i})\ge L(\ga_{n,i+1})$ for any $i\ge1$.\\
	Up to subsequence and a diagonal argument we can assume that $\ga_{n,i}\to\ga_i$ as $n\to\infty$ uniformly. Then we denote $(\Ga_n)=\cup_i (\ga_{n,i})$ and $(\Ga)=\cup_i (\ga_i)$. Arguing as in the case of $P_S$ we have that $\liminf_n L(\ga_{n,i}) \ge L(\ga_i)$ for any $i$. By Fatou's Lemma we have that
	\beqs
	\begin{split}
		\liminf_n P_C(E_n)&=\liminf_n P(E_n)=\liminf_n \sum_i L(\ga_{n,i}) \ge \sum_i\liminf_n L(\ga_{n,i}) \ge \sum_i L(\ga_i).
	\end{split} 
	\eeqs
	As before, we define a multiplicity function $\te:\R^2\to\N\cup\{+\infty\}$ as $\te=\sum_i \te_i$ with
	\beqs
		\te_i:\R^2\to \N\cup\{+\infty\} \qquad \te_i(x)=\sharp(\ga_i^{-1}(x)).
	\eeqs
	Observe that $L(\ga_i)=\int_{\R^2} \te_i\,d\cH^1$, $\sum_i L(\ga_i)= \int_{\R^2} \te\,d\cH^1$, and the multiplicity functions $\te_i,\te$ are finite $\cH^1$-ae. Arguing as before we want to use the following result.
	
	\begin{lemma}\label{lem4'}
		Let $E,E_n,\ga_i,\ga_{n,i}$ be as in the proof of \eqref{eq7}. Suppose that for some $i$ we have that $\ga_i(t)=x \in \R^2$, $\gamma_i$ is differentiable at $t$, $|B_r(x) \cap E|=0$ for some $r>0$, and $L(\ga_i)>0$. Then $\theta(x) \ge 2$.
	\end{lemma}

	\noindent As in the case of $P_S$, Lemma \ref{lem4'} implies that $\te(x)\ge 2$ at $\cH^1$-ae $x\in (\Ga)\sm\pa^*E$.\\
	In fact, since we have only countably many curves, then $\cH^1(\cup_j \{(\ga_j)\,\,|\,\,L(\ga_j)=0\})=0$. Hence $\cH^1$-ae $x\in (\Ga)\sm\pa^*E$ belongs to a curve $(\ga_i)$ with $L(\ga_i)>0$. Therefore one applies Lemma \ref{lem4'} with such $\ga_i$ exactly as in the above case of $P_S$.\\
	Also, it holds the following result.
	
	\begin{lemma}\label{lem4"}
		Let $E,E_n,\ga_i,\ga_{n,i}$ be as in the proof of \eqref{eq7}. Then for $\cH^1$-ae point $x\in\pa^* E$ it holds that $\te(x)\ge 1$.
	\end{lemma}

	\noindent Therefore, arguing like in \eqref{eq6}, one gets
	\[ \liminf_n P_C(E_n)\ge \int_{\R^2} \theta(x) \,d \cH^1(x) \ge  P(E) + 2\cH^1\Big( (\Gamma)\cap (\R^2\sm \overline{E}) \Big) .\]
	By an argument analogous to the one in Lemma \ref{lem5}, we get that $\big( (\Gamma)\cap (\R^2\sm \overline{E}) \big) \cup \overline{E^1}$ is equivalent $\mbox{mod } \cH^1$ to a connected set. Hence $\cH^1\big( (\Gamma)\cap (\R^2\sm \overline{E}) \big) \ge St(E)$, and thus $P(E) + 2St(E)   \le  \liminf_{n \to +\infty} P_C(E_n)$.\\
\end{proof}

\noindent We conclude this part by proving the lemmas used in the proof of Proposition \ref{propliminf}.

\begin{proof}[Proof of Lemma \ref{lem4}]
	Let us reparametrize $\ga$ so that $\gamma : [-1/2,1/2] \to \R^2$, $t=0$, and without loss of generality $x=0$. Let $\delta>0$ be small enough such that $\gamma_{|[-\delta,\delta]}\con B_r(x)$ is the graph of a $L$-Lipschitz function over its tangent. For $\ep>0$ we define (see also Figure \ref{fig1}):\\
	i) $A = \{ (x_1,x_2)\in \R^2 \ : \ |x_1| \le \delta, \ |x_2| \le L|x_1| \}$,\\
	ii) $A_\ep = \cN_\ep(A)$.\\
	For all $s\in [-\delta,\delta]$, we have $\gamma(s) \in A$. Let us choose $0<\ep<r$ small enough such that $B_\ep(\gamma(-\delta))$, $B_\ep(\gamma(\delta))$ and $B_\ep(x)$ are pairwise disjoint. As $\gamma_n$ converges uniformly to $\gamma$, there exists $n_0$ such that for all $n \ge n_0$, for all $s\in [-\delta,\delta]$, we have $\gamma_n(s) \in A_\ep$.\\
	We claim that for all $N>n_0$ there exist $n_\ep \ge N$ and $s_\ep \in [-1/2,1/2]\sm (-\delta,\delta)$ such that
	\beq\label{eq8}
		 |\gamma_{n_\ep}(s_\ep) - x| < 2\ep .
	\eeq
	By virtue of this claim, for $\ep \to 0$ we can see that $s_\ep$ converges to some $s\in[-1/2,1/2]\sm (-\delta,\delta)$ and $n_\ep \to +\infty$. By uniform convergence we have $\gamma_{n_\ep}(s_\ep) \to \gamma(s)$ and $\gamma(s)=x$ by \eqref{eq8}. So $\gamma(0)=\gamma(s)=x$ and $s\neq 0$, thus $\theta(x) \ge 2$.\\
	Thus we are left to prove the above claim. Suppose by contradiction that for all $n>n_0$ for all $s \in [-1/2,1/2]\sm (-\delta,\delta)$ we have that
	\[ \gamma_n(s) \notin B_{2\ep}(x) .\]
	Let $C_\ep = B_{2\ep}(x) \sm A_\ep$ and denote $C^+_\ep$ and $C^-_\ep$ its two connected components. Since $\gamma_n$ is a closed curve and for $n>n_0$ it holds that $\ga_n|_{[-\de,\de]}\con A_\ep$, then either $C^+_\ep$ or $C^-_\ep$ is contained in $int(\ga_n)=E_n \mbox{ mod }\cL^2$. Therefore
	\[ |E_n \cap B_r(x)| \ge \frac{1}{2}|C_\ep| \]
	for $n>n_0$, but this contradicts the hypotheses.
	
	\begin{figure}[h]
	\begin{center}
		\begin{tikzpicture}[scale=0.7]
		\fill[color=gray!20, opacity=1] (0,0) circle(2) ;
		\fill[color=white, opacity=1]
		(0,1) --
		plot[domain=60:180, samples=50] ({-5+cos(\x)},{3+sin(\x)}) --
		plot[domain=180:300, samples=50] ({-5+cos(\x)},{-3+sin(\x)}) --
		(0,-1) --
		plot[domain=240:360, samples=50] ({5+cos(\x)},{-3+sin(\x)}) --
		plot[domain=0:120, samples=50] ({5+cos(\x)},{3+sin(\x)}) -- cycle ;
		\fill[color=gray!80, opacity=0.6]
		(-5,3) -- (0,0) -- (5,3) -- (5,-3) -- (0,0) -- (-5,-3) -- cycle ;
		\fill[color=gray!80, opacity=0.6]
		(0,1) --
		plot[domain=60:180, samples=50] ({-5+cos(\x)},{3+sin(\x)}) --
		plot[domain=180:300, samples=50] ({-5+cos(\x)},{-3+sin(\x)}) --
		(0,-1) --
		plot[domain=240:360, samples=50] ({5+cos(\x)},{-3+sin(\x)}) --
		plot[domain=0:120, samples=50] ({5+cos(\x)},{3+sin(\x)}) -- cycle ;
		\draw[dashed] (0,0) circle(2) circle(1) ;
		\draw[thick] plot[domain=-5:0, samples=100]	(\x,{0.3*\x*sin(60*\x)}) ;
		\draw[thick] plot[domain=0:5, samples=100]	(\x,{0.04*\x*\x*(3-\x)}) ;
		\draw[dashed] (-5,{0.3*(-5)*sin(60*(-5))}) circle(1) ;
		\draw[dashed] (5,{0.04*5*5*(3-5)}) circle(1) ;
		\draw plot[domain=-4.5:5.5, samples=100]	(\x,{-0.75-(\x/5)^2}) ;
		\draw[->] (-6.5,0) -- (6.5,0) ;
		\draw (-5,0) node{$|$} (0,0) node{$|$} (5,0) node{$|$} ;
		\draw (-5,0.5) node{$-\delta$} (0,0.5) node{$x$} (5,0.5) node{$\delta$} ;
		\draw (-3,1.5) node{$\gamma_{|[-\delta,\delta]}$} ;
		\draw (-4,1) node{$\gamma_{n|[-\delta,\delta]}$} ;
		\draw[->] (-3,1.2) -- (-2.6,0.5) ;
		\draw[->] (-4,0.8) -- (-3.5,-1.15) ;
		\draw (4,1.5) node{$A$} (5.5,1.5) node{$A_\varepsilon$}
		(0,-1.5) node{$C^-_\varepsilon$} (0,1.5) node{$C^+_\varepsilon$} ;
		\draw[<->] (-3.7,-2.2) -- (-3.2,-3) ;
		\draw (-3.2,-2.5) node{$\varepsilon$} ;
		\end{tikzpicture}
	\end{center}
	\caption{Sketch of the construction in the proof of Lemma \ref{lem4}.}\label{fig1}
	\end{figure}
\end{proof}

\textcolor{white}{text}

\begin{proof}[Proof of Lemma \ref{lem5}]
	Without loss of generality we can identify $E_n=int(\ga_n)=\{x\in\R^2\sm(\ga)\,\,|\,\,Ind_{\ga_n}x=1\}$ and we let $int(\ga)=\{x\in\R^2\sm(\ga)\,\,|\,\,Ind_\ga x\equiv1 \mbox{ mod }2\}$. Since $E_n\cup (\gamma_n) \to int(\ga)\cup (\gamma)$ in Hausdorff distance and $E_n\cup (\gamma_n)$ is connected then, by a simple application of Golab theorem, $int(\ga)\cup (\gamma)$ is connected as well.\\
	
	\noindent Step 1: $int(\ga)\cup (\gamma)$ and $E\cup (\gamma)$ are equivalent $ \mbox{ mod }  \cH^1$.\\
	We first prove $int(\ga)\cup (\gamma) \subset E\cup (\gamma)$ up to a $ \cH^1$-negligible set. If $x \in int(\ga)$ then, for $r$ small and $n$ large enough we have that
	\[ \frac{|E_n\cap B_r(x)|}{|B_r(x)|} = 1, \]
	then $x \in E^1 = \mathring{E} \mbox{ mod }\cH^1$. So $int(\ga)\cup (\gamma) \subset E\cup (\gamma)$ up to a $ \cH^1$-negligible set. Second, we prove that $\big( E\cup (\gamma) \big) \sm \big( int(\ga)\cup (\gamma) \big)$ is $ \cH^1$-negligible. Indeed if $x \in \big( E\cup (\gamma) \big) \sm \big( int(\ga)\cup (\gamma) \big)$ then $x \notin \mathring{E}$, otherwise if $B_\ro(x)\con E$, then $B_\ro(x)\con E_n \mbox{ mod }\cL^2$ for $n$ large and eventually $x\in int(\ga)$. So we got that $x \in  \pa E$. Since $ \pa^*E \subset (\gamma)$ as a consequence of Lemma \ref{lem4}, we have $x \notin  \pa^* E$. So $x\in\pa E\sm \pa^* E$, which is $\cH^1$-negligible.\\
	
	\noindent Step 2: $E\cup (\gamma)$ and $\big( (\gamma)\cap ( \R^2\sm \overline{E}) \big) \cup \overline{E^1}$ are equivalent $ \mbox{ mod }  \cH^1$.\\
	We first notice that $ \pa^* E =  \pa^* E^1 \subset  \pa E^1 \subset \overline{E^1}$, then $E^1 \cup  \pa^* E \subset \overline{E^1}$.\\
	a)Let $x\in E\cup (\gamma)$.\\
		\indent	i) If $x \in E$ then $x \notin  \R^2\sm \overline{E} = E^0  \mbox{ mod }  \cH^1$. Therefore $x \in E^1\cup  \pa^* E \subset \overline{E^1}$.\\
		\indent	ii) If $x \in (\gamma)$, then either $x \in E^1\cup  \pa^* E \subset \overline{E^1}$ or $x \in E^0 =  \R^2\sm \overline{E}  \mbox{ mod }  \cH^1$.\\
		\indent So $x \in \big( (\gamma)\cap ( \R^2\sm \overline{E}) \big) \cup \overline{E^1}$ up to a $ \cH^1$-negligible set.\\
	b) Let $x \in \big( (\gamma)\cap ( \R^2\sm \overline{E}) \big) \cup \overline{E^1}$.\\
			\indent i) If $x \in (\gamma)\cap ( \R^2\sm \overline{E})$ then $x \in (\gamma)$.\\
			\indent ii) If $x \in \overline{E^1}$, then either $x \in E^1 = \overset{\circ}{E}  \mbox{ mod }  \cH^1 \subset E$ or $x \in  \pa^* E \subset (\gamma)$ or $x \in E^0$. In this last case \indent as  $\overline{E^1} \subset \overline{E}$ and $\mathring{E} \cap E^0 = \emptyset$, we have $x \in  \pa E =  \pa^* E  \mbox{ mod }  \cH^1 \subset (\gamma)$.\\
		\indent So $x \in E\cup (\gamma)$ up to a $ \cH^1$-negligible set.\\
	
	\noindent Putting together Step 1 and Step 2, we conclude that $\big( (\gamma)\cap ( \R^2\sm \overline{E}) \big) \cup \overline{E^1}$ is equivalent $ \mbox{ mod }  \cH^1$ to $int(\ga)\cup (\gamma)$, which is connected.\\
	The thesis for $\big( (\gamma)\cap \mathring{E} \big) \cup \overline{E^0}$ follows using the same arguments. Such set is equivalent $\mbox{mod }  \cH^1$ to $( \R^2\sm E) \cup (\gamma)$, which is equivalent $ \mbox{mod }  \cH^1$ to $\big( \R^2\sm \overline{int(\ga)}\big) \cup (\gamma)$. This last set is connected as limit in Hausdorff distance of $ \R^2\sm \overline{E_n}$, which is connected since $\overline{E_n}=\overline{int(\ga_n)}$ is simply connected.\\
\end{proof}

\begin{proof}[Proof of Lemma \ref{lem4'}]
	Up to reparametrize $\ga_i$, let $\delta>0$ be small enough such that $\gamma_i{|_{[-\delta,\delta]}}\con B_r(x)$ is the graph of a $L$-Lipschitz function over its tangent. Arguing like in the proof of Lemma \ref{lem4} the following claim holds.\\
	Fix $\ep>0$. Then for all $N$ sufficiently big there exist $n_\ep \ge N$ and $i_\ep \in \N$ such that:\\
	i) if $i_\ep=i$ then there exists $s_\ep \in [-1/2,1/2]\sm (-\delta,\delta)$ such that
		\[ |\gamma_{n_\ep,i}(s_\ep) - x| < 2\ep, \]
	ii) otherwise, there exists $s_\ep \in [-1/2,1/2]$ such that
		\[ |\gamma_{n_\ep,i_\ep}(s_\ep) - x| < 2\ep. \]
	If for a sequence $\ep\to0$ the first alternative holds, the proof follows as in the case of Lemma \ref{lem4}. So let us assume that for $\ep\to0$ the second alternative occurs. Let
	\[
	I_{\ep} = \left\{ j \in \N\sm\{i\} \ : \
	\exists s \in [-1/2,1/2] \quad \text{satisfying} \quad |\gamma_{n_\ep,j}(s) - x| < 2\ep \right\} 
	\]
	Assume without loss of generality that $\theta_i(x) = 1$, otherwise already $\te(x)\ge2$. Then, since $\ga_i$ is also differentiable at $t$, for any $r>0$ it holds that $|int(\gamma_i) \cap B(x,r)| > 0$. Since $|E \cap B(x,r)|=0$, then
	\[
	\big|\big[int(\gamma_i) \cap B(x,r)\big]\sm\big[ \cup_{j \in I_\ep} int(\gamma_{n_\ep,j}) \big]\big|\to0
	\]
	as $n_\ep \to +\infty$. Then, for $N$ large enough, we have
	\[ \sum_{j \in I_\ep} |int(\gamma_{n_\ep,j})| \ge  \frac{1}{2} |int(\gamma_i) \cap B(x,r)| .\]
	By the isoperimetric inequality we have that
	\begin{equation*}
	\sum_{j \in I_ \ep} |int(\gamma_{n_ \ep,j})| \le \frac{1}{4\pi}
	\sum_{j \in I_ \ep} L(\gamma_{n_ \ep,j})^2 \le C_1
	\left( \sup_{j \in I_ \ep} L(\gamma_{n_ \ep,j}) \right) P(E_{n_\ep})\le C_2 \sup_{j \in I_ \ep} L(\gamma_{n_ \ep,j}).
	\end{equation*}
 	Then there exists $j_ \ep \in I_ \ep$ such that
	\[ L(\gamma_{n_ \ep,j_ \ep}) \ge \frac{|int(\gamma_i) \cap B(x,r)|}{4C}>0 .\]
	Since the curves $(\gamma_{n,i})_i$ are ordered so that their length is non-increasing in $i$, then $j_ \ep$ is bounded when $ \ep \to 0$. Hence there is a sequence $\ep\to0$ and some $\overline{j}\neq i$ such that $\ga_{n_\ep,\overline{j}}(s_\ep)\in B_{2\ep}(x)$ for some $s_\ep\to s$. Thus $\ga_{\overline{j}}(s)=x$ and $\te(x)\ge 2$.\\
\end{proof}

\begin{proof}[Proof of Lemma \ref{lem4"}]
	For any $\de>0$ let
	\[
	E_{n,\de}:=E_n\cup\bcup_{L(\ga_{n,i})\le\de} int(\ga_{n,i}).
	\]
	By Lemma \ref{lem1} the boundary decomposition of $E_{n,\de}$ consists of a finite number of curves, independently of $n$. In particular there exists the limit $E_\de=\lim_n E_{n,\de} \supset E$ in the $L^1$ sense. Observe that
	\beqs
		|E_{n,\de}\sm E_n| \le \sum_{L(\ga_{n,i})\le\de} |int(\ga_{n,i})|\le C\sum_{L(\ga_{n,i})\le\de} L(\ga_{n,i})^2\le C\de P(E_n) \le C\de,
	\eeqs
	with $C$ independent of $n$. Hence
	\[
		\bigg|\bcap_{\de>0} E_\de \sm E\bigg| \le |E_\de\sm E| =\lim_n |E_{n,\de}\sm E_n| \le C\de,
	\]
	for any $\de>0$. Then $E=\cap_{\de>0} E_\de$.\\
	Now let $\de_j\searrow0$ and $E_{\de_m}=\cap_{j=1}^m E_{\de_j}$. Then $|E|=\lim_m |E_{\de_m}|$, that is $\|\chi_{E_{\de_m}}\|_{L^1}\to\|\chi_E\|_{L^1}$. Also, since $\cL^2(\pa E)=0$, it is easily verified that $\chi_{E_{\de_m}}\to\chi_E$ pointwise almost everywhere. In particular
	\beq\label{eq28}
		\chi_{E_{\de_m}}\xrightarrow[m\to\infty]{}\chi_E \qquad\mbox{ in $L^1$.}
	\eeq
	From now on let $\ind_\ga(x)$ denote the index of $x$ with respect to a curve $\ga$. Up to reparametrization we can assume that each $\ga_{n,i}$ is positively oriented with respect to $int(\ga_{n,i})=\{x\in\R^2\sm(\ga_{n,i})\,|\,\ind_{\ga_{n,i}}(x)=1\}$. Also call $int(\ga_i):=\{x\in\R^2\sm(\ga_i)\,|\,\ind_{\ga_i}(x)=1\}$. Then by Lemma \ref{lem1} we can write
	\beqs
		\chi_{E_{n,\de}}=\chi_{int(\ga_{n,0})}-\sum_{i=1}^{k(\de)} \chi_{int(\ga_{n,i})},
	\eeqs
	\beq\label{eq29}
	\chi_{E_\de}=\chi_{int(\ga_0)}-\sum_{i=1}^{k(\de)} \chi_{int(\ga_i)}.
	\eeq
	Observe that for any $j\neq l$ it holds that
	\beq\label{eq31}
		|int(\ga_j)\cap int(\ga_l)|=\lim_n |int(\ga_{n,j})\cap int(\ga_{n,l})|=0.
	\eeq
	Hence $f(x):=\sum_{i=1}^\infty \chi_{int(\ga_i)}(x) \in\{0,1\}$ is well defined $\cL^2$-ae. Letting $f_k:=\sum_{i=1}^k \chi_{int(\ga_i)}$, it is easily verified that $f_k\to f$ pointwise $\cL^2$-ae. Also $f_k,f\le \chi_{B_{R}(0)}$ for $R$ sufficiently large, then by Lebesgue Theorem we get that $f_k\to f$ in $L^1$, and thus
	\beq\label{eq30}
		\chi_{int(\ga_0)}- \sum_{i=1}^k \chi_{int(\ga_i)} \xrightarrow[k\to\infty]{}  \chi_{int(\ga_0)}- \sum_{i=1}^\infty \chi_{int(\ga_i)} \qquad\mbox{ in $L^1$.}
	\eeq
	Putting together \eqref{eq28}, \eqref{eq29}, and \eqref{eq30} we conclude that
	\beq
		\chi_E= \chi_{int(\ga_0)}- \sum_{i=1}^\infty \chi_{int(\ga_i)}.
	\eeq
	Finally, for any field $X\in C^1_c(\R^2;\R^2)$, parametrizing each $\ga_i$ by arclength on $[0,L(\ga_i)]$ and using \eqref{eq31} we have that
	\beqs
	\begin{split}
		-\int_{\cF E} X\,dD\chi_E &=\int_E \div X = \int \bigg( \chi_{int(\ga_0)}- \sum_{i=1}^\infty \chi_{int(\ga_i)} \bigg) \div X = \int_{int(\ga_0)} \div X -\sum_{i=1}^\infty \int_{int(\ga_i)} \div X =\\
		&= \int_0^{L(\ga_0)} \lgl X\circ\ga_0, T\tau_0\rgl\,dt - \sum_{i=1}^\infty \int_0^{L(\ga_i)} \lgl X\circ\ga_0, T\tau_0\rgl\,dt=\\
		&=\int X \,d\mu_0 - \sum_{i=1}^\infty \int X \, d\mu_i,
	\end{split}
	\eeqs
	where $T\tau_i$ denotes the clockwise rotation of an angle $\pi/2$ of the tangent vector $\tau_i$ of $\ga_i$, and $\mu_i$ is the vector valued measure
	\beqs
		\mu_i(p)= \left(\sum_{y\in \ga_i^{-1}(p)} T\tau_i(y)\right)\big(\cH^1\res(\ga_i)\big) (p).
	\eeqs
	It follows that $\mu=\mu_0-\sum_{i=1}^\infty \mu_i$ is a measure and
	\beqs
		\mu=-D\chi_E.
	\eeqs
	Since $\mu$ is concentrated on $(\Ga)=\bcup_{i=0}^{\infty}(\ga_i)$ and $D\chi_E$ is concentrated on $\cF E$, it follows that for $\cH^1$-ae point $p\in\pa^* E$ one has that $\te(p):=\sum_{i=0}^\infty \te_i(p)\ge 1$.
\end{proof}

\textcolor{white}{.}

\subsection{Lim\,sup inequality on regular sets}

In this subsection we deal with the $\limsup$ inequality evaluated on smooth bounded sets. It is useful to remember that such sets have a finite number of connected components and holes.

\begin{lemma}\label{lem:FiniteTree}
	Let $E$ be a bounded smooth set. Let $S$ and $S^c$ be Steiner trees of $\overline{E^1}$ and $\overline{E^0}$, respectively. Then:\\
	i) the Steiner trees $S$ and $S^c$ are finite,\\
	ii) if $v$ is a vertex of $S$ or $S^c$ and $v\in\pa E$, then $v$ is an endpoint and the edge having $v$ as endpoint is orthogonal to $\pa E$.
\end{lemma}

\begin{proof}
	i) Let $S_k$ be a connected component of $S$. By regularity properties of Steiner trees (\cite{PaSt13}), $S_k$ has at most one endpoint on each connected component of $E$. Then $S_k$ has a finite number of endpoints $\{p_1,\cdots,p_N\}$, therefore $S_k$ is a Steiner tree for $K=\{p_1,\cdots,p_N\}$. Hence $S_k$ is then a finite tree. Moreover $S_k$ connects at least two distinct connected components of $E$. Then, by minimality, there exists only a finite number of connected components of $S$. Thus $S$ is a finite tree. The same argument can be applied to $S^c$.\\
	ii) Let $v\in\pa E\cap S$ be a vertex of $S$, which is a finite tree. Then $v$ clearly has degree $1$, otherwise another edge with endpoint at $v$ would intersect $\mathring{E}$. The orthogonality follows immediately from the first variation of the length of the edge having endpoints $v$ and $w$, keeping $w$ fixed and $v\in\pa E$.
\end{proof}

\begin{remark}
	Let $E$ be a bounded smooth set. Let $S$ be the Steiner tree of $\overline{E^1}$ and let $S_\ep:=\cN_\ep(S)$. Then
	\beq\label{eq10}
		\lim_{\ep\to0} |S_\ep|=0, \qquad \qquad \limsup_{\ep\to0} \cH^1\left(\{ x\,\,|\,\,d(x,S)=\ep \}\right)\le 2\cH^1(S).
	\eeq
	In order to obtain \eqref{eq10} recall that $S$ is a finite tree (Lemma \ref{lem:FiniteTree}), and thus we can assume that $S=\overline{S}$ is closed. Hence $S$ has finitely many connected components $S_i$, and each $S_i$ is a connected compact finite tree given by the union of finitely many essentially disjoint segments with positive length, i.e. $S_i=\bcup_{j=1}^{J_i} s_j$ and $\cH^1(s_j\cap s_k)=0$ for $j\neq k$. If $s_j$ is a segment, it is easy to check that
	\[
	\lim_{\ep\to0} |\cN_\ep(s_j)|=0, \qquad \lim_{\ep\to 0} \cH^1\left(\left\{ x\,\,|\,\,d(x,s_j)=\ep\right\}\right)=2\cH^1(s_j),
	\]
	for any $j=1,...,J_i$. For $\ep$ small enough we have that
	\[
	|\cN_\ep(S_i)|\le \sum_{j=1}^{J_i} |\cN_\ep(s_j)|, \qquad  \cH^1\left(\{ x\,\,|\,\,d(x,S)=\ep \}\right)\le \sum_{j=1}^{J_i} \cH^1\left(\left\{ x\,\,|\,\,d(x,s_j)=\ep\right\}\right),
	\]
	and thus \eqref{eq10} follows passing to the limit $\ep\to0$ using the fact that there are only finitely many connected components $S_i$.\\
	We observe that since $\cH^1(S)=\cH^1(\overline{S})$ and $\overline{S}$ is $1$-rectifiable, \eqref{eq10} also follows by applying Theorem 3.2.39 in~\cite{Fe69}.
\end{remark}

\begin{prop}\label{proplimsup1}
	Let $\hat{E}$ be a bounded smooth set. Then there exists a sequence $\tilde{E}_\ep$ of bounded connected smooth sets such that
	\beq
		\tilde{E}_\ep\xrightarrow[\ep\to0]{}\hat{E} \qquad\mbox{ in }L^1,
	\eeq
	\beq
		\limsup_{\ep\to 0} P_C(\tilde{E}_\ep) \le P(\hat{E}) + 2St(\hat{E}).
	\eeq
\end{prop}

\begin{proof}
	Let $S$ be the Steiner tree of $\overline{\hat{E}^1}$ and let $S_\ep=\cN_\ep(S)$. Define $E_\ep=E\cup S_\ep$. The set $S$ is a finite tree with endpoints on $\pa E$ and such that every other vertex is a triple point where edges meet forming three angles equal to $\frac{2}{3}\pi$ (\cite{PaSt13}). Hence for $\ep$ small enough the set $E_\ep$ is connected, indecomposable, and there exist finitely many points $p_1,...,p_k\in\pa E_\ep$ such that $\pa E_\ep\sm\{p_1,..,p_k\}$ is smooth. Hence one can clearly approximate $E_\ep$ by bounded connected smooth sets $E_{\ep,m}$ with $|E_{\ep,m}\De E_\ep|<\frac{1}{m}$ and $|P(E_{\ep,m})-P(E_\ep)|<\frac{1}{m}$. By a diagonal argument and using \eqref{eq10} we get the desired sequence $\tilde{E}_\ep$.
\end{proof}

\begin{prop}\label{proplimsup2}
	Let $\hat{E}$ be a bounded smooth set. Then there exists a sequence $\tilde{F}_\ep$ of bounded simply connected smooth sets such that
	\beq
	\tilde{F}_\ep\xrightarrow[\ep\to0]{}\hat{E} \qquad\mbox{ in }L^1,
	\eeq
	\beq
	\limsup_{\ep\to 0} P_S(\tilde{F}_\ep) \le P(\hat{E}) + 2St(\hat{E}) + 2St_c(\hat{E}).
	\eeq
\end{prop}

\begin{proof}
	Let $S,S^c$ be the finite Steiner trees of $\overline{\hE^1},\overline{\hE^0}$. We can assume that $S,S^c$ are closed. Let us define
	\beqs
		U_\ep=\overline{\cN_\ep(S)}\sm \mathring{\hat{E}}, \qquad \qquad U_\ep^c=\cN_\ep(S^c)\cap \overline{\hE}.
	\eeqs
	Let also
	\beqs
		\tilde{E}_\ep=\big(\overline{\hE}\sm U^c_\ep\big)\cup U_\ep,
	\eeqs
	which is closed. Suppose $\ep$ is sufficiently small so that if $A,B$ are two connected components of $S$ (or of $S^c$), then $\overline{\cN_\ep(A)}\cap \overline{\cN_\ep(B)}=\epty$. We can also assume that if $\cN_\ep(A_S)\cap \cN_\ep(B_{S^c})\neq\epty$ for two connected components $A_S\con S$ and $B_{S^c}\con S^c$, then $A_S\cap B_{S^c}=\{v\}\neq\epty$ where $v$ is an endpoint of both $S$ and $S^c$. Observe that $\si(\pa \hat{E})=\cH^1(S)+\cH^1(S^c)$, where $\si(\pa \hE)$ is the infimum of the Steiner problem of $\pa \hat{E}$. By \eqref{eq10} we have that $\tE_\ep\to\hE$ in $L^1$ sense and $\limsup_\ep P(\tE_\ep)\le P(\hE)+2\si(\pa\hE)=P(\hE)+2St(\hat{E}) + 2St_c(\hat{E})$ as $\ep\to0$.\\
	Now we modify $\tE_\ep$ in order to obtain $\tF_\ep$ preserving $L^1$ convergence to $\hE$ and $\limsup$ estimate on the perimeters. More precisely, we want to regularize $\pa\tE_\ep$ around its finitely many corners, i.e. the points of $\pa\tE_\ep$ at which $\pa\tE_\ep$ is not smooth. This will lead us to a simple smooth curve which will be $\pa\tF_\ep$. Observe that the vertices of $S,S^c$ are only endpoints or triple points, and if a point $v$ is a vertex of both $S$ and $S^c$ then $v\in\pa\tE_\ep$, $v$ is an endpoint of both $S$ and $S^c$, and both the edge of $S$ and $S^c$ with endpoint at $v$ are orthogonal to $\pa\hE$ at $v$.\\
	Any corner $p$ of $\pa\tE_\ep$ corresponds to a vertex $v$ of $S$ or $S^c$, in the sense that, for $\ep$ small, $p\in B_{2\ep}(v)$ for a unique vertex $v$. We call edges of $\pa\tE_\ep$ the smooth curves having as endpoints two corners of $\pa\tE_\ep$. We want to change $\tE_\ep$ modifying such edges around the singular points corresponding to a given vertex $v$. More precisely, given a vertex $v$ we modify the edges $\si_k$ inside $B_{2\ep}(v)$ according to the following instructions.\\
	1) Let $v\in S$ be a triple point of $S$. Then modify inside $B_{2\ep}(v)$ the six edges of $\pa\tE_\ep$ corresponding to the three singular points $p_1,p_2,p_3$ associated to $v$ by smoothing the corners around $p_1,p_2,p_3$. Leave those edges unchanged out of $B_{2\ep}(v)$. Also modify $\tE_\ep$ correspondingly. \\
	2) Let $v\in S\sm S^c$ be an endpoint of $S$. Then modify inside $B_{2\ep}(v)$ the four edges of $\pa\tE_\ep$ corresponding to the two singular points $p_1,p_2\in \pa\tE_\ep$ associated to $v$ by smoothing the corners around $p_1,p_2$. Leave those edges unchanged out of $B_{2\ep}(v)$. Also modify $\tE_\ep$ correspondingly. See also Figure \ref{fig2} on the left.
	
		\begin{figure}[h]
		\begin{center}
			\begin{tikzpicture}[scale=2.5]
			\fill[color=gray!50]
			plot[domain=-60:60] ({cos(\x)},{sin(\x)}) -- cycle;
			\fill[color=gray!50]
			(0.5,-0.1)--(2,-0.1)--(2,0.1)--(0.5,0.1)--cycle;
			\draw
			plot[domain=-60:-5.711] ({cos(\x)},{sin(\x)});
			\draw
			plot[domain=5.711:60] ({cos(\x)},{sin(\x)});
			\draw
			(0.995,0.1)--(2,0.1);
			\draw
			(2,-0.1)--(0.995,-0.1);
			\filldraw[black]
			(1,0)circle(0.03);
			\path
			(0.8,0)node{$v$};
			\draw[dashed, line width=1pt]
			plot[domain=30:60] ({cos(\x)},{sin(\x)}) (0.866,0.5)to[out=-60, in=180](1.3,0.1)--(2,0.1);
			\draw[dashed, line width=1pt]
			plot[domain=-60:-30] ({cos(\x)},{sin(\x)}) (0.866,-0.5)to[out=60, in=180](1.3,-0.1)--(2,-0.1);
			\end{tikzpicture}\qquad\qquad\qquad
			\begin{tikzpicture}[scale=0.6]
			\fill[color=gray!50]
			(-3, 0.5)--(0, 0.5)--(0,1)to[out=90, in=-80](-0.35,4)--(-3,4)--(-3, 0.5);
			\fill[color=gray!50]
			(-3,- 0.5)--(0,- 0.5)--(0,-1)to[out=270, in=80](-0.35,-4)--(-3,-4)--(-3,- 0.5);
			\fill[color=gray!50]
			(0, 0.5)--(3, 0.5)--(3,- 0.5)--(0,- 0.5)--(0, 0.5);
			\draw
			(-3,0.5)--(0,0.5)--(0,1)to[out=90, in=-80](-0.35,4);
			\draw
			(-3,-0.5)--(0,-0.5)--(0,-1)to[out=270, in=80](-0.35,-4);
			\draw
			(3,-0.5)--(0,-0.5)--(0,0.5)--(3,0.5);
			\filldraw[black]
			(0,0)circle(0.1);
			\path
			(-0.65,0)node{$v$};
			\draw[dashed, line width=1pt]
			(-0.35,4)to[out=-80, in=90](0,1)to[out=-90, in=180](0.5,0.5)--(3,0.5);
			\draw[dashed, line width=1pt]
			(-3,0.5)--(-1.5,0.5)to[out=0, in=180](1.5,-0.5)--(3,-0.5);
			\draw[dashed, line width=1pt]
			(-3,-0.5)--(-0.5,-0.5)to[out=0, in=90](0,-1)to[out=270, in=80](-0.35,-4);
			\path
			(3.5,0.5)node{$s_1$};
			\path
			(3.5,-0.5)node{$s_2$};
			\path
			(0.1,4)node{$\si$};
			\end{tikzpicture}
		\end{center}
	\caption{The two cases of $v$ endpoint of $S$ in the proof of Proposition \ref{proplimsup2}: on the left $v\in S\sm S^c$, on the right $v\in S\cap S^c$. The gray area denotes $\tE_\ep$. The continuous lines denote $\pa\tE_\ep$, the dashed lines denote the modifications smoothing the corners.}\label{fig2}
	\end{figure}
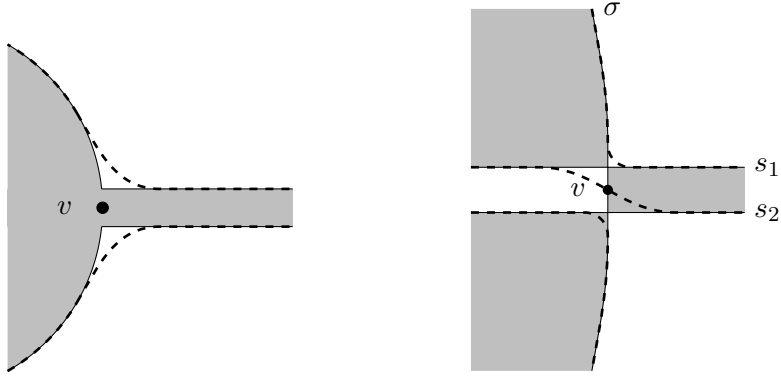
	
	\noindent 3) Let $v\in S\cap S^c$ be endpoint of both $S$ and $S^c$. Since both the edges of $S$ and $S^c$ having $v$ as endpoint are orthogonal to $\pa\hE$, around $v$ the boundary $\pa\tE_\ep$ is determined by two parallel segments $s_1,s_2$ together with a third curve $\si\con\pa\hE$ meeting once each segment (see Figure \ref{fig2} on the right) at the two corners $p_1,p_2$ corresponding to $v$. Independently of the choice of $s_1$ or $s_2$, desingularize $\pa\tE_\ep$ by modifying the edges as depicted in Figure \ref{fig2} on the right. More precisely, parametrizing $\si\cap B_{2\ep}(v)$ with constant velocity on $[0,1]$, we can say that $\si$ splits $s_1$ (and $s_2$) into two parts $s_{1,l},s_{1,r}$ (and $s_{2,l},s_{2,r}$) respectively on the left or on the right of the parametrization of $\si$. So delete the part of $\si$ between the two intersections $p_1,p_2$, connect smoothly $s_{1,l}$ with $s_{2,r}$, and then desingularize the remaining two corners joining one piece of $\si$ with $s_{1,r}$ and the other piece of $\si$ with $s_{2,l}$ without crossing (see Figure \ref{fig2} on the right).\\
	4) Let $v\in S^c\sm S$ be a vertex. Modify the edges corresponding to $v$ by the same rules of points 1) and 2).\\
	
	\noindent Now call $\tF_\ep$ the resulting set. By construction $\pa \tF_\ep$ is smooth, hence
	\[
	\pa\tF_\ep=\sqcup_{i=1}^K (\si_i),
	\]
	for a finite number of smooth closed simple curves $\si_i$. We want to prove that $K=1$, so that $\tF_\ep$ is the interior of a smooth closed simple curve, and thus $\tF_\ep$ is simply connected and then the proof is completed.\\
	Let $J^\pm_j$ be the finitely many curves given by Theorem \ref{thm1} applied to $\hE$. Call
	\beqs
		E_j=int(J^+_j), \qquad H_{-j}=\overline{int(J^-_j)},
	\eeqs
	for any possible $j$. If $A$ is a connected component of $\mathring{\hat{E}}$, then we can write $A=E_j\sm \sqcup_{i=1}^r H_{-j_i}$ for some $j,j_i$.\\
	We claim that $A\sm S^c$ is simply connected.\\
	In fact $A$ is homeomorphic to $B\sm\{p_1,...,p_r\}$, where $B$ denotes the open ball in $\R^2$ and $p_1,...,p_r\in B$. Also $A\sm S^c$ is homeomorphic to $B\sm T$, where $T$ is a closed planar graph without cycles with vertices at points $V_T=\{q_1,...,q_l,p_1,...,p_r,t_1,...,t_s\}$, where $q_i\in\pa B$ are endpoints and $t_i\in B$ are triple points. Therefore $A\sm S^c$ is homeomorphic to $B\sm \sqcup_{i=1}^l L_i$ where $L_i\simeq[0,1]$ is an embedded curve contained in $\overline{B}$ with $L_i\cap \pa B= \{q_i\}$. Hence $A\sm S^c$ is simply connected.\\
	By the above claim, for $\ep$ small, also $A\sm U^c_\ep=A\sm\big(\cN_\ep(S^c)\cap A\big)$ and the latter is homeomorphic to $ A\sm S^c$, and such sets are simply connected.\\
	Consider now $S_1^A,...,S_{R_A}^A$ the finitely many connected components of $S$ which are connected to $A\sm U^c_\ep$ (observe that these are not all the connected components of $S$ touching $\overline{A}$, but these are the connected components of $S$ having endpoints on $\overline{A}$ which are not endpoints of $S^c$). For any $i=1,...,R_A$ the set $\cN_\ep(S^A_i)\sm\overline{\hE}$ is homeomorphic to $B$. Also each $\cN_\ep(S^A_i)\sm\mathring{{\hE}}$ is simply connected. Hence by construction the open set
	\beq\label{eq12}
		V_A:=int\bigg[A\sm U_\ep^c \cup \bigsqcup_{i=1}^{R_A} \big( \cN_\ep(S_i^A)\sm \mathring{{\hE}} \big)\bigg],
	\eeq
	where $int(\cdot)$ denotes the interior of a set $(\cdot)$, is homeomorphic to $B$. By construction, for $\ep$ sufficiently small the finitely many connected components of $\mathring{\tE}_\ep$ are either a finite union of sets the form $V_A,V_{A'}$ having in common some $S^A_i=S^{A'}_j$, or they are of the form
	\beq
		\cN_\ep(S_m)\sm \overline{\hE},
	\eeq
	where $S_m$ is a connected component of $S$ such that each endpoint of $S_m$ is also an endpoint of $S^c$. In any case each connected component of $\mathring{\tE}_\ep$ is homeomorphic to $B$. Also the closed set $\tE_\ep$ is connected, and the closures of two connected components of $\mathring{\tE}_\ep$ are either disjoint or they intersect exactly in two points which are corners of $\pa\tE_\ep$ corresponding to a vertex $v\in S\cap S^c$ as represented in Figure \ref{fig2} on the right.\\
	Hence the finitely many modifications on the boundary $\pa\tE_\ep$ by construction lead to a simply connected smooth set $\tF_\ep$, and the proof is completed.
\end{proof}

\textcolor{white}{.}

\subsection{Approximation}

Here we want to prove that a set of finite perimeter $E$ with $\cH^1(\pa E\De \pa^* E)=0$ can be approximated by a sequence of smooth sets verifying the suitable $\limsup$ inequalities.

\begin{prop}\label{propapprox}
	Let $E$ be an essentially bounded set of finite perimeter satisfying $\pa E=\pa^* E \mbox{ mod }\cH^1$. Then there exist a sequence $\hE_\ep$ of bounded smooth sets of finite perimeter such that\\
	\beq
	\hE_\ep\xrightarrow[\ep\to0]{}E,
	\eeq
	\beq
	\limsup_{\de\to 0} P(\hE_\ep) + 2St(\hE_\ep) \le P(E) + 2St(E),
	\eeq
	\beq
	\limsup_{\de\to 0} P(\hE_\ep) + 2St(\hE_\ep) + 2St_c(\hE_\ep) \le P(E) + 2St(E) + 2St_c(E).
	\eeq
\end{prop}

\begin{proof}
	By Remark \ref{rem1} we can assume that $E=\mathring{E}$. Fix $\ep\in(0,1)$. Adopt the following notation.
	\begin{itemize}
		\item $Y_i$, for $i \in I$, are the components of $E$ given by Theorem \ref{thm2}.
		\item $T_{i,j}$, for $j \in J_i$, are the holes of $Y_i$.
		\item $J_{i,\ep}\con J_i$ is a subset such that
		\begin{itemize}
			\item $J_i \sm J_{i,\ep}$ is finite,
			\item $ \sum_{j\in J_{i,\ep}} P(T_{i,j}) < \ep^2$,
		\end{itemize}
		i.e. $J_{i,\ep}$ contains the indexes of the small holes of $Y_i$.
		\item $\tilde{Y}_i = Y_i \cup  \left( \bigcup_{j \in J_{i,\ep}} T_{i,j} \right)$ is the filling of the small holes of $Y_i$.
		\item $I_\ep = \{ i \in I \ : \ |\tilde{Y}_i| > 2\ep \}$ are the indexes of the not too small sets $\tY_i$.
		\item For $i\in I_\ep$ the set $Y_{i,\ep}$ is the smooth open set approximating $\tY_i$ from within as given by Theorem \ref{thmSchmidt} with respect to the parameter $\delta = \ep^3$ (this is possible by Remark \ref{rem2} since $\tY_i$ is indecomposable with a finite number of holes by construction).
		\item $N_\ep = \sharp(I_\ep)$.
		\item $ \hE_\ep = \bigcup_{i \in I_\ep} Y_{i,\ep}$.
	\end{itemize}
	We need to show that such $\hE_\ep$ satisfies the thesis.\\
	
	\noindent Since for any $j_0\in J_{i_\ep}$ it holds that $P(T_{i,j_0})\le \sum_{j \in J_{i,\ep}} P(T_{i,j}) < \ep^2<1$, by isoperimetric inequality we have that
	\beq\label{eq13}
		\sum_{j \in J_{i,\ep}} |T_{i,j}| \le C_{iso} \sum_{j \in J_{i,\ep}} P(T_{i,j})^2 \le  C_{iso} \sum_{j \in J_{i,\ep}} P(T_{i,j}) \le  C_{iso} \ep^2,
	\eeq
	where $C_{iso}=\frac{1}{4\pi}$ is the isoperimetric constant in dimension $2$. Observe that $\ep<1<1/C_{iso}$. If $i \in I_\ep$, then \eqref{eq13} implies that $2\ep < |\tilde{Y}_i| = |Y_i|	+  \sum_{j \in J_{i,\ep}} |T_{i,j}| \le |Y_i| + C_{iso}\ep^2$, and thus $|Y_i|>\ep$. Hence
	\beq\label{eq16}
	\ep N_\ep \le \sum_{i \in I_\ep} |Y_i| \le \sum_{i \in I} |Y_i| \le |E| .
	\eeq
	Since $P(\tY_i)\le P(Y_i)$ it holds that
	\beq\label{eq14}
	 \begin{split}
	P(\hE_\ep) & \le  \sum_{i \in I_\ep} P(Y_{i,\ep})
	 \le   \sum_{i \in I_\ep} \Big( P(\tilde{Y}_i) + \ep^3 \Big)  \le   \sum_{i \in I_\ep} P(Y_i) + \ep^3N_\ep
	 \le   \sum_{i \in I} P(Y_i) + \ep^3N_\ep \le \\
	& \le P(E) + \ep^3N_\ep . 
	\end{split}
	\eeq\\
	
	\noindent We claim that
	\beq\label{eq24}
		 \hE_\ep\xrightarrow[\ep\to0]{} E \qquad\mbox{ in $L^1$}.
	\eeq
	In fact let us estimate
	\beq\label{eq17}
	|E \Delta \hE_\ep| \le
	\sum_{i \in I\sm I_\ep} |Y_i| +
	\sum_{i \in I_\ep} |\tilde{Y}_i \Delta Y_{i,\ep}| +
	\sum_{\substack{i \in I_\ep \\ j \in J_{i,\ep}}} |T_{i,j}| .
	\eeq
	Since $\sum_{i \in I} |Y_i|=|E|$, then
	\beq\label{eq20}
		\lim_{\ep\to0} \sum_{i \in I\sm I_\ep} |Y_i|=0.
	\eeq
	By Lemma \ref{lem2} we can assume that any $\tY_i$ is open and $P(\tY_i)=\cH^1(\pa\tY_i)$. Hence, since by Theorem \ref{thmSchmidt} we have that $\tilde{Y}_i \Delta Y_{i,\ep} \subset \mathcal{N}_{\ep^3}(\pa \tilde{Y}_i)$, it follows either by a direct argument or using Theorem 3.2.39 in \cite{Fe69} that
	\beq\label{eq15}
		|\tilde{Y}_i \Delta Y_{i,\ep}|\le |\mathcal{N}_{\ep^3}(\pa \tilde{Y}_i)| \le  2(1+\ep)\ep^3 \cH^1(\pa \tilde{Y}_i) = 2(1+\ep)\ep^3 P(\tilde{Y}_i) \le 4\ep^3 P(E),
	 \eeq
	for $\ep\le\ep(i)$ depending on $i$.\\
	Moreover $ \sum_{i \in I_\ep} |\tilde{Y}_i \Delta Y_{i,\ep}| \le |sat(E)| < +\infty$, then	$ \limsup_{\ep \to 0} \sum_{i \in I_\ep} |\tilde{Y}_i \Delta Y_{i,\ep}| < +\infty$. We denote $\ep_h$ a subsequence such that
	\[ \limsup_{\ep \to 0} \sum_{i \in I_\ep} |\tilde{Y}_i \Delta Y_{i,\ep}| =
	\lim_{h \to +\infty} \sum_{i \in I_{\ep_h}} |\tilde{Y}_i \Delta Y_{i,\ep_h}| . \]
	Since $|\tY_i\De Y_{i,\ep}|\to0$ for any $j$, then for all $\eta>0$ there exists $H>0$ such that for all $h>H$ it holds that $ \sum_{i \in I_{\ep_h}\sm I_{\ep_H}} |\tilde{Y}_i \Delta Y_{i,\ep_h}| < \eta/2$. Since $I_{\ep_H}\con I_{\ep_h}$ and $I_{\ep_H}$ is finite, by \eqref{eq15} we can write that
	\[
	 \sum_{i \in I_{\ep_H}} |\tilde{Y}_i \Delta Y_{i,\ep_h}| \le 4\ep_h^3N_{\ep_H} P(E) \le  4\ep_h^3N_{\ep_h} P(E),
	\]
	for any $\ep\le\min\{\ep(i)\,\,|\,\,i\in I_{\ep_H}  \}$.\\
	Taking into account \eqref{eq16} we can choose $h$ large enough so that $4N_{\ep_h}\ep_h^3 P(E) < \eta/2$ and then
	\[
	 \sum_{i \in I_{\ep_h}} |\tilde{Y}_i \Delta Y_{i,\ep_h}| \le
	\sum_{i \in I_{\ep_H}} |\tilde{Y}_i \Delta Y_{i,\ep_h}|
	+ \sum_{i \in I_{\ep_h} \sm I_{\ep_H}} |\tilde{Y}_i \Delta Y_{i,\ep_h}| < \eta .
	\]
	Then
	\[ \limsup_{\ep \to 0} \sum_{i \in I_\ep} |\tilde{Y}_i \Delta Y_{i,\ep}| =
	\lim_{h \to +\infty} \sum_{i \in I_{\ep_h}} |\tilde{Y}_i \Delta Y_{i,{\ep_h}}| < \eta .\]
	Thus, taking $\eta \to 0$, we have
	\beq\label{eq18}
	 \sum_{i \in I_\ep} |\tilde{Y}_i \Delta Y_{i,\ep}| \xrightarrow[\ep\to0]{}0.
	\eeq
	 Finally by \eqref{eq13} we have
	 \beq\label{eq19}
	 \sum_{\substack{i \in I_\ep \\ j \in J_{i,\ep}}} |T_{i,j}| \leqslant C_{iso} \ep^2 N_\ep \xrightarrow[\ep\to0]{} 0. 
	 \eeq
	 Putting together \eqref{eq20}, \eqref{eq18}, \eqref{eq19}, and \eqref{eq17} we obtain the claim \eqref{eq24}.\\
	 
	 \noindent Let $S$ be a Steiner tree of $\overline{E^1}$. We denote $S^*_k$ for $k \in I^S$ a connected component of
	 \[
	 	 S \cup \left( \bigcup_{i \in I\sm I_\ep} \pa \tilde{Y}_i \right)
	  \]
	 such that there exist at least two distinct indexes $i,j \in I_\ep$ such that $S^*_k$ connects $Y_i$ and $Y_j$. Also let
	 \[
	 I^S_k = \{ i \in I \sm I_\ep \ | \ \pa \tilde{Y}_i \subset S^*_k \},
	 \]
	 and denote by $S_{k,j}$ a connected component of $S$ contained in a given $S^*_k$.\\
	 We now prove that
	 \beq\label{eq21}
	 \sharp(I^S)\le \sharp(I_\ep)=N_\ep.
	 \eeq
	 In fact, by minimality, for any couple $(i,j)\in I_\ep\times I_\ep$ with $j>i$ there exists at most one $k\in I^S$ such that $S^*_k$ connects $Y_i$ and $Y_j$.  We define a function $\chi:\{(i,j)\in I_\ep\times I_\ep\,|,j>i\}\to \{0,1\}$ such that $\chi(i,j)=1$ if and only if there exists (a unique) $k\in I^S$ such that $S^*_k$ connects $Y_i$ and $Y_j$. Up to relabeling we can suppose that $\chi(1,2)=1$. By construction $\sharp I^S\le \sharp\chi^{-1}(1)$. Since $\chi(1,2)=1$, then by minimality at most one of the values $\chi(1,3)$ and $\chi(2,3)$ is equal to $1$; that is $\sum_{i=1}^2 \chi(i,3)\le1$. Iterating this argument one has that
	 \beqs
	 	\sum_{i=1}^{j-1} \chi(i,j)\le 1
	 \eeqs
	 for any $j=2,...,N_\ep$. And this implies that $\sharp\chi^{-1}(1)\le \sum_{j=2}^{N_\ep} \sum_{i=1}^{j-1} \chi(i,j)\le N_\ep$, and we have \eqref{eq21}.\\	 
	 
	 \noindent Now for any $S^*_k$ let $I_{k,\ep} = \{ \al \in I_\ep \ | \ S^*_k \text{ is connected to } Y_\al \}$. For $\al \in I_{k,\ep}$, since $\tilde{Y}_\al \Delta Y_{\al,\ep} \subset \mathcal{N}_{\ep^3}(\pa \tilde{Y}_\al) \subset \mathcal{N}_{\ep^3}(\pa Y_\al)$, there exists a segment $s_{\al,k}$ with length less than $\ep^3$ connecting $S^*_k$ and $Y_{\al,\ep}$. Given $S^*_k$, denote by $S_{k,\ep}$ the union
	 \beqs
	 	S_{k,\ep}=\bigcup_j S_{k,j}\cup\bcup_{i\in I^S_k} \pa\tY_i\cup\bcup_{\al\in I_{k,\ep}} s_{\al,k}.
	 \eeqs
	Then
	\beq
		\cH^1(S_{k,\ep}) \le \sum_j\cH^1(S_{k,j}) + \sum_{i \in I^S_k} P(\tilde{Y}_i) + \sharp(I_{k,\ep}) \ep^3 .
	\eeq
	
	\noindent Define also $I^2_\ep = \{ (i,j) \in I_\ep \times I_\ep \ | \ i \neq j, \ \overline{Y_i}\cap\overline{Y_j} \neq \emptyset \}$. Similarly as before, if $(i,j)\in I^2_\ep$ there exists a segment $S_{ij,\ep}$ connecting $Y_{i,\ep}$ and $Y_{j,\ep}$ such that $\cH^1(S_{ij,\ep})\le 2\ep^3$.\\
	
	\noindent Finally define
	\beq
		S_ \ep = \left( \bigcup_{k \in I^S} S_{k, \ep} \right) \cup
		\left( \bigcup_{(i,j) \in I^2_ \ep} S_{ij, \ep} \right) .
	\eeq
	By construction the set $S_\ep\cup\bcup_{i\in I_\ep} Y_{i,\ep}=S_\ep\cup \hE_\ep$ is connected. Then
	\beq\label{eq22}
	\begin{split}
		St(E_ \ep) & \le \cH^1(S_ \ep) \le
	 \sum_{k\in I^S} \cH^1(S_{k, \ep}) + \sum_{(i,j)\in I^2_ \ep} \cH^1(S_{ij, \ep}) \le \\
	& \le   \sum_{k\in I^S}\sum_j \cH^1(S_{k,j}) +
	\sum_{k\in I^S} \sum_{i \in I^S_k} P(\tilde{Y}_i) +
	\sum_{k \in I^S}  \ep^3 \sharp(I_{k, \ep}) + 2 \ep^3 \sharp(I^2_ \ep) \le\\
	& \le  \cH^1(S) + \sum_{i \in I\sm I_ \ep} P(\tilde{Y}_i) +
	 \ep^3\sharp(I^S)\sharp(I_ \ep) + 2 \ep^3 \sharp(I_ \ep^2) \le\\
	& \le \cH^1(S) +  \sum_{i \in I\sm I_ \ep} P(Y_i) +  \ep^3N_ \ep^2 + 2 \ep^3N_ \ep^2 .
	\end{split}
	\eeq
	
	\noindent Analogously let $S^c$ be a Steiner tree of $\overline{E^0}$. Now let
	\[ 
	S^c_\ep = S^c \cup \left( \bigcup_{\substack{i \in I_\ep \\ j \in J_{i,\ep}}} \pa T_{i,j} \right) .
	\]
	Observe that if $T_{i,j}$ is a hole of $Y_i$ then either it is filled in $\tilde{Y}_i$, or it merges with $ext(\hE_\ep)$, or it is included in a hole of $Y_{i,\ep}$. Thus if $H_{l,\ep}$ are the holes of $\hE_\ep$, then $S_\ep^c\cup\bcup_l H_{l,\ep}$ is connected with the exterior $ext(\hE_\ep)$. Therefore
	\beq\label{eq23}
	St_c(\hE_\ep) \le \cH^1(S_\ep^c) \le \cH^1(S^c) + \sum_{i \in I_\ep} \sum_{j \in J_{i,\ep}}P(T_{i,j}) \le \cH^1(S^c) + \ep^2 N_\ep .
	\eeq
	
	\noindent Putting together \eqref{eq14}, \eqref{eq22}, and \eqref{eq23} we obtain
	\beqs
	\begin{split}
		\limsup_{\ep\to0} P(\hE_\ep) + 2St(\hE_\ep) &\le \limsup_{\ep\to0} P(E)+\ep^3 N_\ep + 2\bigg(\cH^1(S) +  \sum_{i \in I\sm I_ \ep} P(Y_i) +  \ep^3N_ \ep^2 + 2 \ep^3N_ \ep^2\bigg)= \\
		&= P(E)+ 2 St(E),
	\end{split}
	\eeqs
	and
	\beqs
	\begin{split}
		\limsup_{\ep\to0} &P(\hE_\ep) + 2St(\hE_\ep) +2 St_c(\hE_\ep)\le \\ &\le \limsup_{\ep\to0} P(E)+\ep^3 N_\ep + 2\bigg(\cH^1(S) +  \sum_{i \in I\sm I_ \ep} P(Y_i) +  \ep^3N_ \ep^2 + 2 \ep^3N_ \ep^2\bigg)+ 2\bigg(  \cH^1(S^c) + \ep^2 N_\ep  \bigg)=\\
		&= P(E)+2St(E)+2St_c(E).
	\end{split}
	\eeqs
	Taking into account \eqref{eq24}, we see that $\hE_\ep$ satisfies the thesis.
\end{proof}

\textcolor{white}{text}

\section{Application: a liquid drop model with connectedness constraint}

\noindent In the end, we want to discuss an explicit application of the energies $\overline{P_C}$, $\overline{P_S}$. More precisely we point out how such energies used in place of the usual perimeter can give existence of a solution to a minimization problem.\\
\noindent Fix $\al\in(0,2)$ and $m>0$. We consider the following minimization problem
\beq\label{eq25}
	\min\bigg\{ P(E) + \int_{E\times E} \frac{1}{|x-y|^\al}\,dx\,dy \quad\big|\quad E\con\R^2 \mbox{ measurable}, \,\,\,|E|=m \bigg\}.
\eeq
which is sometimes called {\it Gamow's liquid drop model}. This problem, introduced in \cite{Ga} 
in three dimensions and for $\alpha=1$, 
has been studied for instance in \cite{KnMu13} (see also \cite{GoNoRu15,MuNoRu18}), where it is proven that there exist
two threshold values $m_1(\al),m_2(\al)$ such that:\\
i) for all $m \le m_1(\alpha)$, \eqref{eq25} has a solution,\\
ii) for all $m > m_2(\alpha)$, \eqref{eq25} has no solution.

\noindent We will prove now that, substituting $P$ with $\overline{P_C}$ or $\overline{P_S}$ in 
\eqref{eq25}, there exists a solution to the new minimization problem for {\it any} $\al\in(0,2),m>0$. This is clearly a mathematical tool in order to avoid the non-existence phenomenon happening for $m>m_2(\alpha)$, which is essentially due to the lack of compactness of $\R^2$. However the use of $\overline{P_C}$ or $\overline{P_S}$ in place of $P$ can be also seen as a model for charged liquid drops which cannot split. We are not aware of any physical situation of this kind, but material science is always in progress!

\begin{lemma}\label{lem6}
	The map $\displaystyle \E \mapsto \iint_{E\times E} \frac{d x d y}{|x-y|^\alpha}$
	is continuous with respect to the convergence in $L^1_{loc}(\R^2)$.
\end{lemma}

\begin{proof}
	The proof immediately follows from the following observations.\\
	i) The function $f(x,y)=\frac{1}{|x-y|^\al}$ belongs to $L^1_{loc}(\R^4)$.\\
	ii) For any $A,B\con\R^2$ it holds that
		\[ (A\times A) \Delta (B\times B) = \Big( (A\smallsetminus B)\times A \Big) \cup
		\Big( (A\cap B)\times (A\Delta B) \Big) \cup \Big( (B\smallsetminus A)\times B \Big) .\]
	iii) By ii), if $E_n,E\con\R^2$ then $\cL^4((E_n\times E_n)\Delta (E\times E)) \le |E_n\De E|\big(|E_n|+|E|+|E_n\cap E|\big)$.\\
	iv) We can estimate
		\[ \left| \iint_{E_n\times E_n} \frac{d x d y}{|x-y|^\alpha}
		- \iint_{E\times E} \frac{d x d y}{|x-y|^\alpha} \right| \leqslant ||f||_{L^1(K\times K)}
		\cL^4\big((E_n\times E_n)\Delta (E\times E)\big) ,\]
		for any $E_n,E\con K\con \R^2$ with $K$ compact.
\end{proof}

\begin{thm}
	For all $\alpha \in (0,2)$ and all $m>0$, the minimization problems
	\beq\label{eq26}
		\min\bigg\{ \overline{P_C}(E) + \int_{E\times E} \frac{1}{|x-y|^\al}\,dx\,dy \quad\big|\quad E\con\R^2 \mbox{ measurable}, \,\,\,|E|=m \bigg\},
	\eeq
	\beq\label{eq27}
		\min\bigg\{ \overline{P_S}(E) + \int_{E\times E} \frac{1}{|x-y|^\al}\,dx\,dy \quad\big|\quad E\con\R^2 \mbox{ measurable}, \,\,\,|E|=m \bigg\}
	\eeq
	admit solutions.
\end{thm}

\begin{proof}
Fix $\al,m$ and define
\[
 \cF_C(E) = \overline{P_C}(E) + \iint_{E\times E} \frac{d x d y}{|x-y|^\alpha} . 
\]
Let $(E_n)$ be a minimizing sequence for the problem \eqref{eq26}, so that in particular $|E_n|=m<+\infty$. Then $\overline{P_C}(E_n)<+\infty$ and there is a sequence of indecomposable sets $E_{n,k}\xrightarrow[k]{}E_n$ in $L^1$ such that $\lim_k P(E_{n,k})=\overline{P_C}(E_n)$. Thus, by lower semicontinuity of the perimeter, one has $P(E_n)\le\overline{P_C}(E_n)\le \sup_n \cF_C(E_n)<+\infty$. Also by Lemma \ref{lem3} we have that
\[
 2\diam(E^1_n) \leqslant \overline{P_C}(E_n) \leqslant \sup_n \cF_C(E_n) < +\infty .
\]
Up to a translation, we may assume that $0 \in E^1_n$ and then $E_n$ is uniformly essentially bounded. Then, by compactness of $BV$ functions, there exists a limit set $E$ (up to a subsequence) with respect to $L^1$ convergence. In particular $|E|=m$ is a competitor for problem \eqref{eq26}. As $\overline{P_C}$ is lower semicontinuous and $f$ is continuous by Lemma \ref{lem6}, we have that $\cF_C$ is lower semicontinuous and then $\inf \cF_C = \cF_C(E)$, and there exists a minimizer of problem \eqref{eq26}.\\
A completely analogous proof also works in the case of Problem \eqref{eq27}.
\end{proof}

\section*{Acknowledgements}
Simon Masnou and Fran\c{c}ois Dayrens acknowledge support from the French National Research Agency (ANR) research grants MIRIAM (ANR-14-CE27-0019) and GEOMETRYA (ANR-12-BS01-0014), and from the LABEX MILYON (ANR-10-LABX-0070) of Universit\'e de Lyon, within the program "Investissements d'Avenir" (ANR-11-IDEX-0007). 

Matteo Novaga and Marco Pozzetta acknowledge support from the INdAM-GNAMPA Project 2019 Problemi geometrici per strutture singolari.



\begin{thebibliography}{}
	
	\bibitem{AmFuPa00} Ambrosio L., Fusco N., Pallara D.: \emph{Functions of Bounded Variation and Free Discontinuity Problems}, Oxford Science Publications (2000).
	
	\bibitem{AmCaMaMo01} Ambrosio L., Caselles V., Masnou S., Morel J.M.: \emph{Connected components of sets of finite perimeter and applications to image processing}, Journal of the European Mathematical Society, 3(3):39-92 (2001).
	
	\bibitem{BoLeSa15} Bonnivard M., Lemenant A., Santambrogio F.: \emph{Approximation of length minimization problems among compact connected sets}, SIAM J. Math. Anal., 47(2):1489-1529 (2015). 
	 
	\bibitem{BoLeMi18} Bonnivard M., Lemenant A., Millot V.: \emph{On a phase field approximation of the planar Steiner problem: existence, regularity, and asymptotic of minimizers}, Interfaces and free Boundaries, Volume 20, Issue 1, 2018, pp. 69-106.
	
	\bibitem{DoLeWo17} Dondl P.W., Lemenant A., Wojtowytsch S.: \emph{Phase field models for thin elastic structures with topological constraint}, Arch. Rational Mech. Anal. (2017) 223:693.
	
	\bibitem{DoMuRo} Dondl P.W., Mugnai L., R{\"o}ger M.: \emph{A Phase Field Model for the Optimization of the {W}illmore Energy in the Class of Connected Surfaces}, SIAM J. Math. Anal., 46(2):1610--1632 (2014).

	\bibitem{DoNoWiWo18} Dondl P.W., Novaga M., Wirth B., Wojtowytsch S.: \emph{Approximation of the relaxed perimeter functional under a connectedness constraint by phase-fields}, SIAM J. Math. Anal., 51(5):3902--3920 (2019).
	
	\bibitem{Fe69} Federer H.: \emph{Geometric Measure Theory}, Springer-Verlag Heidelberg New York (1969).
	
	\bibitem{Ga} Gamow G.: \emph{Mass defect curve and nuclear constitution}, Proc. Roy. Soc. London A, 126:632-644 (1930). 
	
	\bibitem{GiPo68} Gilbert E.N., Pollak H.O.: \emph{Steiner minimal trees}, SIAM J. Appl. Math., 16:1-29 (1968).
	
	\bibitem{GoNoRu15} Goldman M., Novaga M., Ruffini B.: \emph{Existence and stability for a non-local isoperimetric model of charged liquid drops}, Arch. Rational Mech. Anal., 217(1):1-36 (2015).
	
	\bibitem{KnMu13} Kn\"{u}pfer H., Muratov C.: \emph{On an isoperimetric problem with a competing nonlocal term I: The planar case}, Comm. Pure Appl. Math., 66(7):1129-1162 (2013).
	
	\bibitem{MuNoRu18} Muratov C., Novaga M., Ruffini B.: \emph{On equilibrium shapes of charged flat drops}, Comm. Pure Appl. Math., 71(6):1049-1073, (2018). 
	
	\bibitem{PaSt13} Paolini E., Stepanov E.: \emph{Existence and regularity results for the Steiner problem}, Calc. Var. Partial Differential Equations, 46(3):837-860 (2013).

 \bibitem{Seifert} Seifert U.: \emph{Configurations of fluid membranes and vesicles}, Adv. Phys., 46(1):13--137 (1997).

 	\bibitem{Po19} Pozzetta M.: \emph{A varifold persepctive on the $p$-elastic energy of planar sets}, Preprint \url{https://arxiv.org/abs/1902.10463} (2019), to appear on J. Convex Analysis.
	
	\bibitem{Sc15} Schmidt T.: \emph{Strict interior approximation of sets of finite perimeter and functions of bounded variation}, Proceedings of the AMS, 143(5):2069-2084 (2015).
	
\end{thebibliography}
\end{document}